\documentclass[sigconf]{acmart}

\usepackage{enumerate}
\usepackage{color}
\usepackage{url}
\usepackage{subfigure}
\usepackage{graphicx,epstopdf}
\usepackage{float}
\usepackage{amsmath,  amssymb, algorithmic, algorithm, url, color, bbm, enumitem}
\usepackage[maxfloats=30]{morefloats}

\newtheorem{assumption}{Assumption}

\newcommand{\hide}[1]{}
\newcommand{\diffd}{\delta}
\newcommand{\bx}{\mathbf{x}}

\usepackage{ifthen}
\newboolean{showcomments}
\setboolean{showcomments}{true}
\newcommand{\yue}[1]{\ifthenelse{\boolean{showcomments}}
{ \textcolor{red}{(Yue says:  #1)}}{}}
\newcommand{\josh}[1]{\ifthenelse{\boolean{showcomments}}
{ \textcolor{red}{(Josh says:  #1)}}{}}
\newcommand{\zhenhua}[1]{\ifthenelse{\boolean{showcomments}}
{ \textcolor{red}{(Zhenhua says:  #1)}}{}}
\newcommand{\addcites}[0]{\ifthenelse{\boolean{showcomments}}
{ \textcolor{green}{(add citation(s))}}{}}
\newcommand{\addcite}[0]{\ifthenelse{\boolean{showcomments}}
{ \textcolor{green}{(add citation(s))}}{}}
\newcommand{\addref}[0]{\ifthenelse{\boolean{showcomments}}
{ \textcolor{green}{(add ref)}}{}}
\newcommand{\todo}[1]{\ifthenelse{\boolean{showcomments}}
{ \textcolor{red}{(To do:  #1)}}{}}

\newboolean{showfixes}
\setboolean{showfixes}{true}
\newcommand{\fixes}[1]{\ifthenelse{\boolean{showfixes}}
{\textcolor{black}{#1}}{}}

\newboolean{showremoved}
\setboolean{showremoved}{false}
\newcommand{\removed}[1]{\ifthenelse{\boolean{showremoved}}
{\textcolor{gray}{#1}}{}}

\newboolean{showrevised}
\setboolean{showrevised}{true}
\newcommand{\revised}[1]{\ifthenelse{\boolean{showfixes}}
	{\textcolor{black}{#1}}{}}

  \makeatletter

\graphicspath{{images/}}

\usepackage{booktabs} 

\copyrightyear{2017}
\acmYear{2017}
\setcopyright{acmlicensed}
\acmConference{e-Energy '17}{May 16-19, 2017}{Shatin, Hong Kong}\acmPrice{15.00}\acmDOI{http://dx.doi.org/10.1145/3077839.3077842}
\acmISBN{978-1-4503-5036-5/17/05}

\begin{document}
\title{Harnessing Flexible and Reliable Demand Response Under Customer Uncertainties}


\author{Joshua Comden}
\affiliation{\institution{Stony Brook University}}
\email{joshua.comden@stonybrook.edu}

\author{Zhenhua Liu}
\affiliation{\institution{Stony Brook University}}
\email{zhenhua.liu@stonybrook.edu}

\author{Yue Zhao}
\affiliation{\institution{Stony Brook University}}
\email{yue.zhao.2@stonybrook.edu}

\begin{abstract}
Demand response (DR) is a cost-effective and environmentally friendly approach for mitigating the uncertainties in renewable energy integration by taking advantage of the flexibility of customers' demands. 
However, existing DR programs suffer from either low participation due to strict commitment requirements or not being reliable in voluntary programs.
In addition, the capacity planning for energy storage/reserves is traditionally done separately from the demand response program design, which incurs inefficiencies.
Moreover, customers often face high uncertainties in their costs in providing demand response, which is not well studied in literature.

This paper first models the problem of joint capacity planning and demand response program design by a stochastic optimization problem, which incorporates the uncertainties from renewable energy generation, customer power demands, as well as the customers' costs in providing DR.
We propose online DR control policies based on the optimal structures of the offline solution. 
A distributed algorithm is then developed for implementing the control policies without efficiency loss. We further offer enhanced policy design by allowing flexibilities into the commitment level. 
We perform real world trace based numerical simulations. Results demonstrate that the proposed algorithms can achieve near optimal social costs, and significant social cost savings compared to baseline methods. 



\end{abstract}

%
%




\maketitle

\section{Introduction}
One of the major issues with the integration of renewable energy sources into the power grid is the increased uncertainty and variability that they bring~\cite{sgdoe}. 
The limited capability to accurately predict this variability makes it challenging for the load serving entities (LSEs) to respond to it~\cite{nrelvariability}. 
If this variability is not sufficiently addressed, it will limit the further penetration of renewables into the grid and even result in blackouts~\cite{doegridintegration}.

\begin{figure*}[!ht]
	\begin{center}
		\subfigure[Load (kW)]{{\includegraphics[width=0.60\columnwidth]{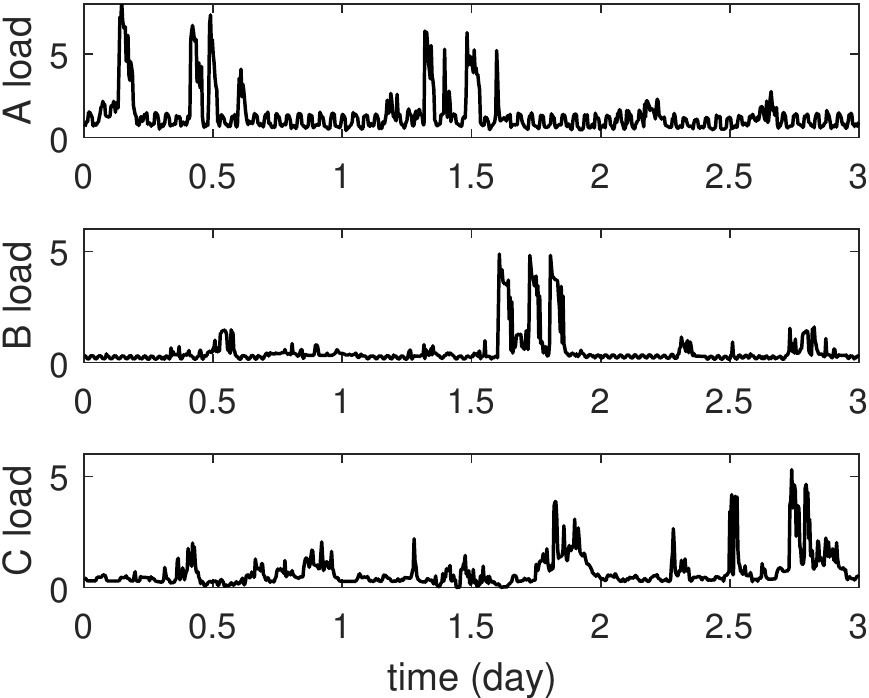}}
			\label{FIG:HOMESABCSAMPLE}}
		\subfigure[Cumulative distribution functions]{{\includegraphics[width=0.60\columnwidth]{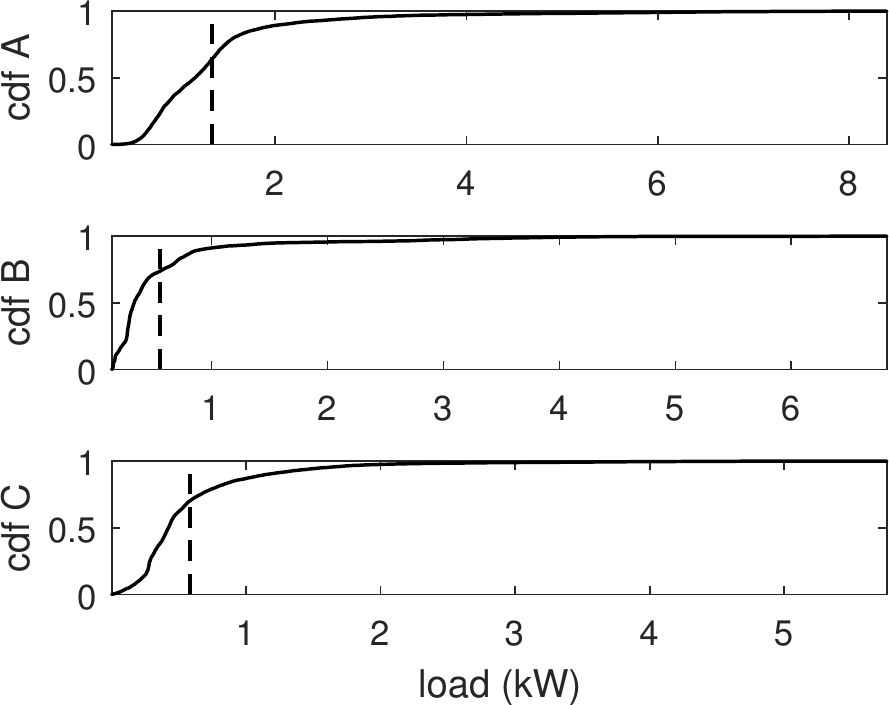}}
			\label{FIG:HOMESABCEMPCDF}}
		%
		%
		\subfigure[395 Homes]{{\includegraphics[width=0.60\columnwidth]{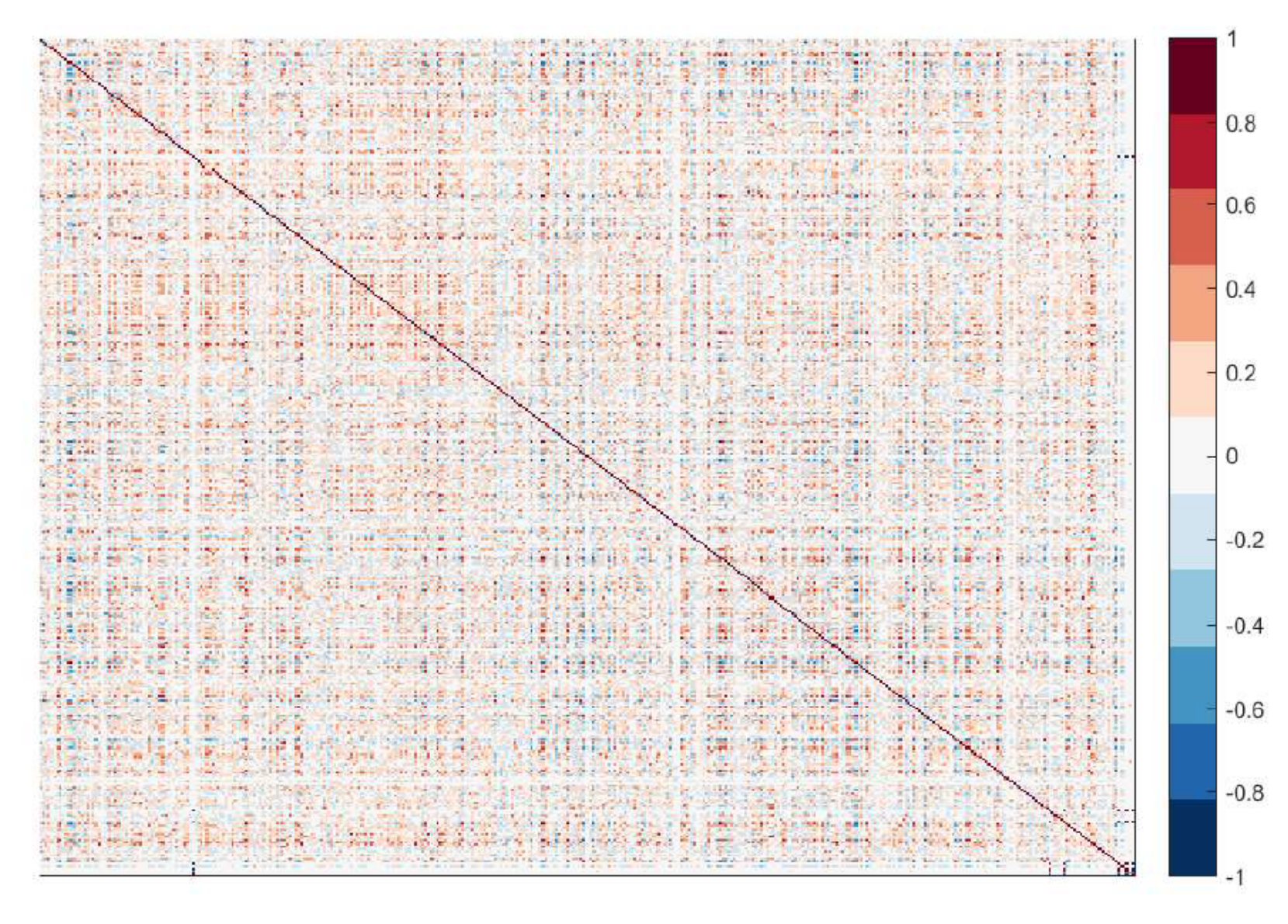}}
			\label{FIG:395CORR}}
		\vspace{-0.15in}
		\caption{For Homes A, B, and C: (a) Load (kW) trace in five-minute intervals from May 6-8, 2012, (b) Cumulative distribution of the five-minute loads (kW) of the 33 days along with the means (dashed lines). (c) Heatmap showing the correlation coefficient matrix for one day of loads from 395 buildings.}
		\label{f.DAHOMES}
	\end{center}
\end{figure*}

\begin{figure*}[ht]
	\begin{center}
		\subfigure[Load mismatch (MW) and price]{{\includegraphics[width=0.59\columnwidth]{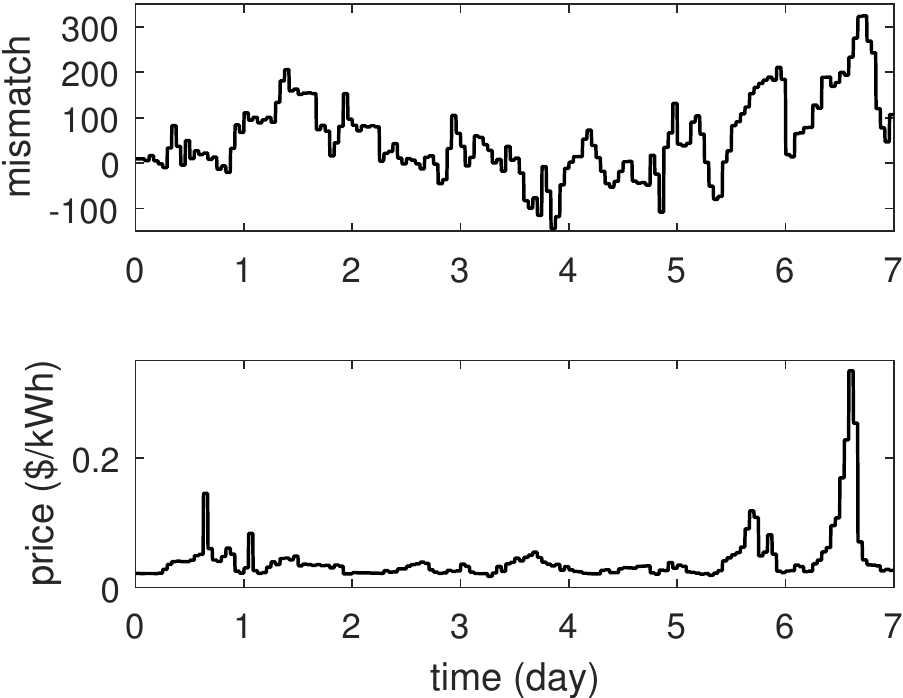}}
			\label{FIG:LSE7DAYS}}
		\subfigure[Cumulative distribution functions]{{\includegraphics[width=0.59\columnwidth]{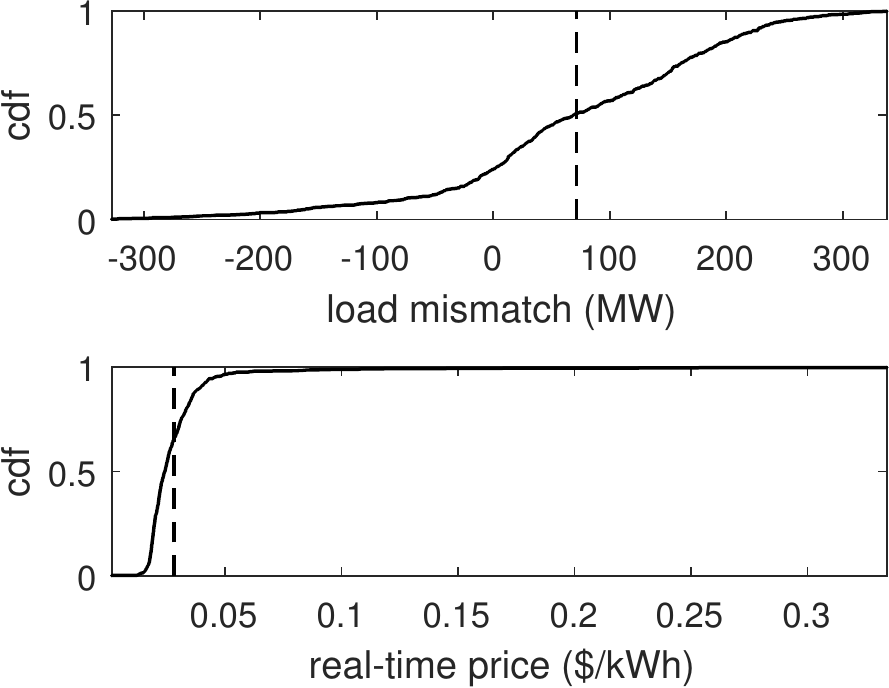}}
			\label{FIG:LSEEMPCCF}}
		\subfigure[Homes A, B, and C]{{\includegraphics[width=0.59\columnwidth]{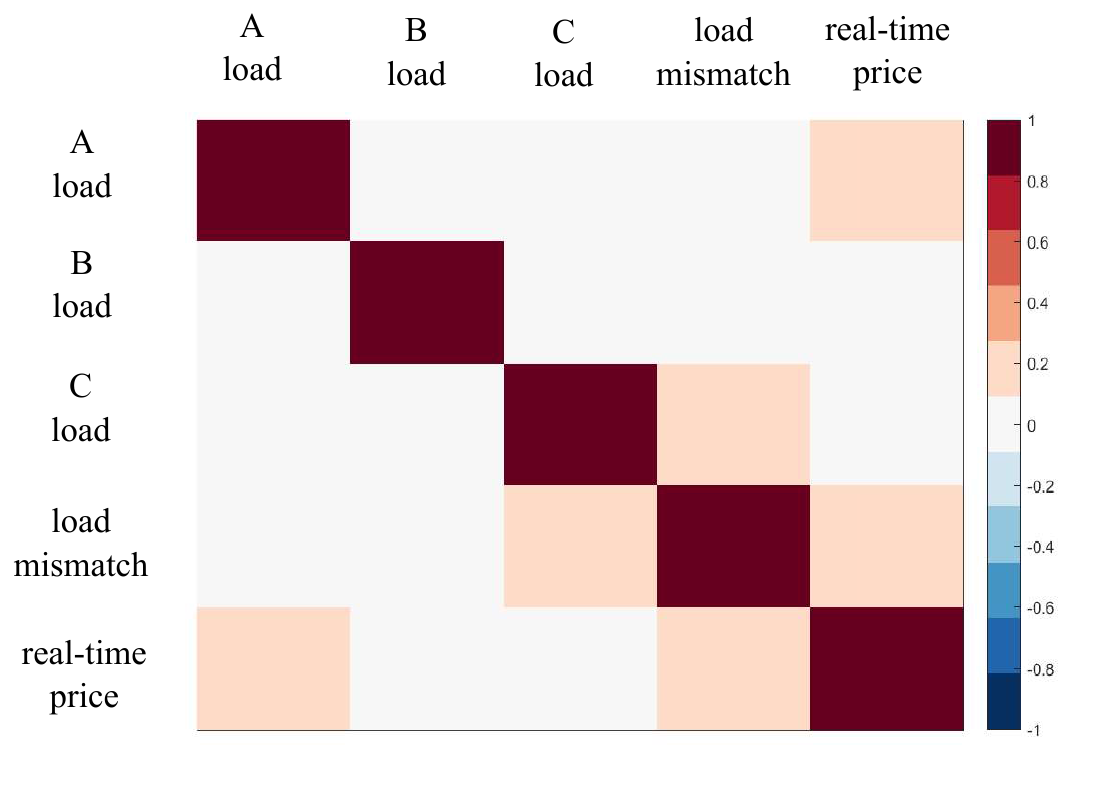}}
			\label{FIG:HOMESABCCORR}}
		\vspace{-0.15in}
		\caption{(a) The ISO-NE load mismatch (MW) between real-time load and day-ahead provisioned load for the Western Massachusetts zone, and the real-time market price from May 23-29, 2012. For the same data during a period of 33 days: (b) Cumulative distribution functions for the load mismatch of ISO-NE, and real-time price,
			(c) Heatmap showing the correlation coefficient matrix for of loads from homes A, B, C, the load mismatch of ISO-NE, and real-time price.}
		\label{f.DALSE}
	\end{center}
\end{figure*}

Various approaches have been implemented or proposed to address this issue. These include improving renewable generation forecast \cite{Pinson2013}, aggregating diverse renewable sources \cite{ZQRGP14}, fast-responding reserve generators, energy storage \cite{bitar2011role, CRZJG16}, and demand response (DR) \cite{vardakas15}, among others.
In particular, in 2013, the California state legislature enforced a solution by passing a bill that requires 1,325 MW of grid energy storage by 2020~\cite{cpucrules, calilaw} and that declares ``Expanding the use of energy storage systems can assist ... in integrating increased amounts of renewable energy resources into the electrical transmission and distribution grid" \cite{calilaw}.  
In order for this solution to be cost-effective, the price of storage needs to be within the range of \$700-750/kWh.
However, in 2013 when the law was passed, prices were about three times that amount~\cite{silverstein13}. 

Compared to energy storage, demand response has advantages to provide reserves to the LSEs in a cost-effective and environmentally friendly way~\cite{frameworkroadmap,vardakas15}.
\fixes{Despite the great potential, there are cases where the increase in the amount of \emph{reliable} DR is much slower than that of renewable integration~\cite{wierman14}, e.g. California's move towards more grid-level energy storage.}
\removed{Despite the great potential, the increase in the amount of \emph{reliable} DR is much slower than that of renewable integration~\cite{wierman14}, partially evidenced by the California's move towards more grid-level energy storage.}
There are multiple reasons about this, but the level of DR commitment is an important factor.

Roughly speaking, there are two types of DR programs based on how much commitment customers need to make in the electric load reduction. 
In the first type, customers are required to make \emph{full} commitment in load reduction, e.g., regulations service~\cite{kirby2005frequency}, capacity bidding~\cite{capacitybiddingreport}.
As a result, it was hoped that the committed demand response can be used as a ``virtual'' energy storage to the power grid, and therefore significantly reduce or at least delay the purchase of additional energy storage. 
However, in such programs, customers have to take all the responsibilities of managing their uncertainties in meeting the hard commitment. 
As highlighted in Section~\ref{sec:cust-uncertain}, customers actually face significant uncertainties when making their decisions, so it is not surprising that the participation level in such committed programs is not high. 

In the latter type, customers do not need to make any commitments, and therefore are willing to participate in DR programs. Examples include emergency demand response programs \fixes{(NYISO \cite{emergencydrpNYISO})} and coincident peak pricing~\cite{CPP}.
The drawback, however, is that from the LSE's perspective this sort of ``voluntary'' demand response is not reliable or sufficiently dispatchable. 
As a result, the LSE still has to heavily rely on energy storage devices.
Readers can refer to Section~\ref{sec:bg} for more background information.

The tradeoff between commitment levels and reliability of demand response raises the following question: \textbf{how can we effectively incentivize the amount of \emph{reliable} demand response?}
This paper moves towards answering this question by making the following main contributions:

\begin{enumerate}[leftmargin=*]
\item We model the social cost minimization problem using stochastic optimization, and characterize the optimal solution in Sections~\ref{sec:model} and \ref{sec:char}. There are two novel features in our stochastic optimization model. First, the uncertainties on the customers' costs to provide DR are explicitly modeled. Second, the capacity planning for the amount of energy storage/reserve needed is jointly optimized with the demand response program design. 
\item Motivated by the optimal structures of the offline solution, we propose simple contracts between customers and LSE and corresponding DR control policies, namely, PRED and LIN, in Section~\ref{sec:alg}.
\fixes{These contracts incentivize customers to participate in DR by offering payments larger than their associated costs.}
We further design a distributed algorithm with guaranteed convergence to overcome the challenge in LIN that LSE may not have enough information about customers' cost functions, so that the policy remains practical in this case. To fully respect and also exploit customer uncertainties, we introduce flexible commitment levels into LIN by allowing limited violation of the contract in Section~\ref{sec:fcdr}.
\item Using real world traces, we evaluate and demonstrate the benefits of our proposed contracts/control policies in Sections~\ref{sec:linear-evaluation}. Our study demonstrates that a) it is essential to take into account the customers' uncertainties into DR program design, and b) optimizing capacity provisioning jointly with DR program design reduces social cost significantly. In particular, the following key insights are obtained: 
\begin{itemize}[leftmargin=*]
\item Simple control policies as we proposed can perform closely to the a-posteriori optimum, and greatly outperform benchmarks similar to the current practice. 
\item The required amount of energy storage/reserve capacity can be significantly reduced due to deeper extraction of DR resources. 
\item Optimizing flexible commitment levels takes into account customer uncertainties even better, and can further reduce the social cost to almost its fundamental limit (a-posteriori optimum). Moreover, as it brings benefits to both the customers and the LSE, both sides are incentivized to participate in the program. 
\end{itemize}
\end{enumerate}


%
%



\section{Background}
\label{sec:bg}

In this section, we first demonstrate customers' uncertainties as well as the uncertainties that LSEs experience by investigating real world load and power market data. 
We then provide an overview of existing demand response programs categorized based on their levels of commitment. 
Namely, there are two major categories: fully committed DR programs, voluntary DR programs, as well as other programs in between such as voltage regulation services~\cite{pjmweb}. 
We then discuss the advantages and drawbacks of these DR programs in the context of customers with significant uncertainties.

\subsection{Customer Uncertainties}
\label{sec:cust-uncertain}



We analyze the Smart$^*$ Data Set obtained from the University of Massachusetts Trace Repository~\cite{barker2012smart} to demonstrate customer consumption uncertainties. 
The specific data we use include the load data from three different homes located in Western Massachusetts given in one-second intervals from 33 days between May 1, 2012 through June 11, 2012, as well as the loads of another 395 buildings for one complete day. 
We average them into five-minute intervals and use these for our trace-based numerical studies. 
\fixes{Five-minute intervals is the data granularity required by the ISO-NE~\cite{ISONEdrguide2011} which is common among other DR programs~\cite{FERCdr5min}.}
A concurrent three day sample of the three homes is given in Figure \ref{FIG:HOMESABCSAMPLE}, which shows peak loads
are non-overlapping in many cases. 
Figure \ref{FIG:HOMESABCEMPCDF} displays the empirical cumulative distribution function of the three homes over the 33 days along with their means.
The large peak-to-mean ratios indicate the significant uncertainties in customers energy consumptions. 

Furthermore, it is observed that customers' loads are only weakly correlated, if at all.
The northwest corner of Figure \ref{FIG:HOMESABCCORR} shows a heatmap of the correlation matrix between the three homes which exhibits very little correlation.
Similarly, for the loads of the other 395 buildings, 
a heatmap of the correlation matrix between 395 
buildings
is shown in Figure \ref{FIG:395CORR} which gives evidence that most customers have loads which are only weakly correlated. 

\subsection{Load Serving Entity Uncertainties}
We analyze the following data from the ISO New England for the Western Massachusetts load zone~\cite{isonezonalinformation} to demonstrate the market uncertainties LSE experience in their daily operation. 
The specific data is given in one-hour intervals of the real-time total load, load contracted in the day-ahead market, and the real-time market price for the same days as the previously described customer data. 
We calculate the real-time load mismatch of supply and demand as the difference between the real-time load and day-ahead market provisioned load.
Figure \ref{FIG:LSE7DAYS} displays a one week sample of the load mismatch and its corresponding real-time price. 
The empirical cumulative distribution functions of the load mismatch and the real-time price for the 33 days are shown in Figure \ref{FIG:LSEEMPCCF}. The high peak-to-mean ratios indicate significant uncertainties in the load mismatch and prices that LSEs need to handle in real time. 
In addition, the southeast corner of Figure \ref{FIG:HOMESABCCORR}
shows that they are weakly correlated over the sampled 33 days.




\begin{figure}
	\includegraphics[width=0.85\columnwidth]{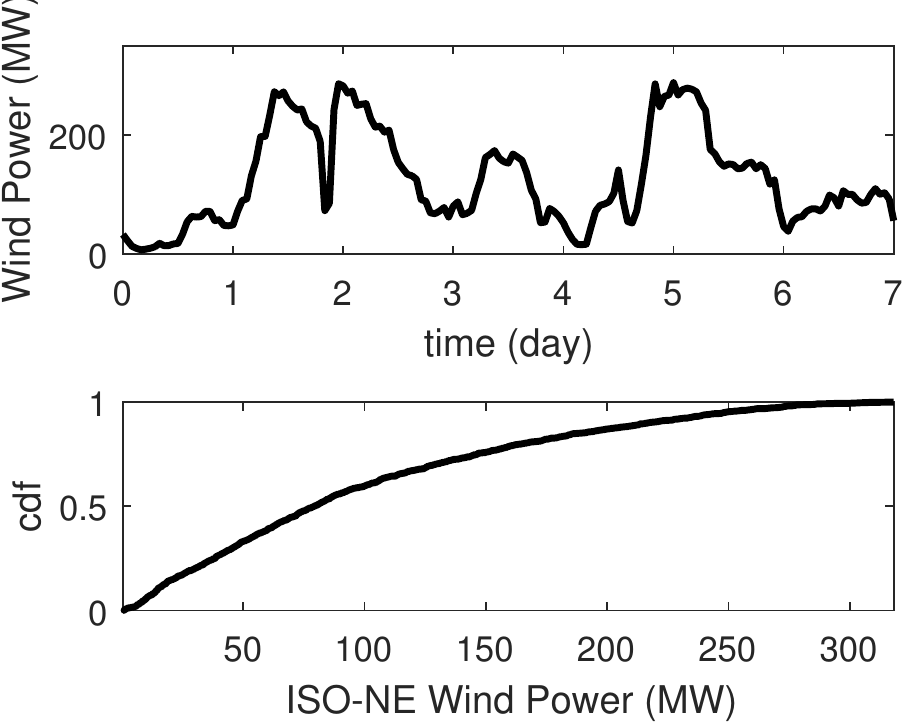}
	\caption{Hourly wind power data taken from the ISO-NE for the same dates as the UMass Homes: A one-week sample of the trace is given along with the empirical cumulative distribution of all 51 days.}
	\label{f.Wind_trace}
\end{figure}

\subsection{A Dichotomy of Existing DR Programs}
Existing DR programs in the current practice can be categorized by the commitment levels of the program participants. Most DR programs either demand the participants to fully commit to responding to DR signals, or allow the participants to not respond without any penalty at all.
In what we refer to as \emph{fully committed} DR programs, customers must respond to DR signals, and have to pay penalties if they fail to do so.
Examples include CAISO BIP (Base Interruptible Program) \cite{caisobip}, ERCOT ERS (Emergency Response Service) \cite{ercot}, and NYISO SCR (Special Case Resources) \cite{nyiso}.
On the other hand we refer to \emph{voluntary} DR programs as ones whose customers only voluntarily respond to DR signals in which ignoring DR signals is penalty-free.
Examples include ERCOT VLR (Voluntary Load Response) \cite{ercot}, NYISO EDRP (Emergency Demand Response Program) \cite{nyiso}, and CAISO DBP (Demand Bidding Program) \cite{caisodbp}.

In between these two extremes with regard to customer commitment are \emph{voltage regulation services} programs (e.g. PJM voltage regulation services~\cite{pjmweb,chenASPDAC}) which have relatively high payments for their fast control speed to follow a signal from the LSE~\cite{aikema2012data, chen2015optimizing}.
The most interesting part of regulation services programs is the flexibility in them which motivates our LIN$^+(\rho)$ program described in Section \ref{sec:fcdr}. 
Instead of requiring customers to strictly follow the signal, there is a predefined violation frequency limit $(100-\rho)\%$, i.e., customers can violate the signal up to $(100-\rho)\%$ of the time.
Therefore, customers do not need to pay the high penalty even if they have some violations from the signal, as long as the probability of violation is below the predefined level.
Intuitively, this can economically help the customers who may accidentally face high costs of DR and decide not to follow the signals. 

\subsection{Challenges of Customer Uncertainties}
We end this section by highlighting the advantages and drawbacks of different commitment levels of DR programs. For a fully committed DR program, a) it has the advantage of providing guarantees to the LSE in getting the expected DR responses, but b) as such, it does not provide any flexibility to the participating customers in all but responding to DR signals, even though they may unexpectedly face a difficulty in real time in responding. For a voluntary DR program, a) it has the advantage that the customers are protected from their risks of having difficulty to respond to DR signals, but b) the LSE has a hard time of getting any reliability guarantees of the DR they get from such programs. 

As a result, fully committed DR programs tend to enlist an insufficient amount of guaranteed DRs because customers are less willing to sign up  due to their inherent uncertainties. On the other hand, voluntary DR programs face a similar consequence but for a different reason: it's hard for them to obtain guaranteed DRs because customers can freely ignore any DR signals. 

\fixes{Recently, third-party curtailment service providers (CSPs)  and aggregators pooling together DR from several customers emerge as a promising opportunity~\cite{vsikvsnys2016dependency}. 
Our work can be viewed as a solution for CSPs and aggregators to lessen DR commitment uncertainties. For instance, the optimization problem and proposed algorithms can be adopted by CSPs and aggregators to allocate DR to their customers in a more effective way.}

\section{Model}
\label{sec:model}


We consider a discrete-time model with time step duration normalized to 1, such that price and demand changes can be updated within a time slot. There is a (possibly long) 
interval of interest $t\in\{1, 2, ..., T\}$, where $T$ can be one month, one year or even longer.
Each timeslot represents the time needed to make changes to demand, which can be 5 minutes.
This timescale is consistent with similar DR approaches as in \cite{harsha2013framework,bitar12}.


There is an LSE who wishes to procure a total amount $D$ of load change.
This can be load reduction as usual, and has the additional generalizability to handle the case for load increase when too much renewable energy is generated and/or customers demand is lower than predicted.
The LSE serves a set of customers indexed by $i\in\{1, 2, ..., N\}$.
We ignore the power network constraints in this paper.
However, our model and algorithms can be extended with extra effort to incorporate those power network constraints.
In particular, in a power distribution network, the exact convex relaxation and convexification~\cite{low2014convex} can be applied to make the problem convex.
Therefore the approach proposed in this paper can be applied.

\vspace{0.05in}
\noindent\textbf{Customers}

Let $d_i(t)$ be the total power demand of the appliances controlled by customer $i$ at time $t$. 
To model the power demand uncertainty for customers we let $d_i(t)$ be a random variable with lower and upper bounds $\underline{d}_i$ and $\overline{d}_i$ respectively.
The LSE will predict customers' demand to purchase power beforehand, e.g., in the day-ahead market.
We denote by $\hat{d}_i(t)$ the predicted power consumption of customer $i$ at time $t$. 

In the real time, customer $i$ observes its own real power demand $d_i(t)$ in the absence of a DR program and decides by $x_i(t)$ under a particular DR program as the amount of demand change from $d_i(t)$.
Here we use a positive $x_i(t)$ as reduction, and a negative value implies demand increase.
So the actual power consumption is $d_i(t)-x_i(t)$.
\fixes{Essentially, we are using the actual consumption as the baseline to measure DR.
Different baseline models can easily be incorporated with additional complexity.
Determining customer baselines for demand response is an active area of research~\cite{wijaya2014bias}.}


To model the loss of utility caused by the change in power consumption $x_i(t)$ from the original demand $d_i(t)$, we assume there is a cost function $C_i(x_i(t); t)$. 
The function is inherently different for different timeslots.
For instance, the customer may have some emergency tasks to finish at some timeslot, so changing the load is costly for that particular $t$.
\fixes{It is not necessary for the change in demand to be instantaneous and the ramping time would be determined by the specific DR program.
There are newly developed DR recommender systems (e.g. DR-Advisor \cite{behl2016dr}) that can be used by homeowners to manage DR commitments.}

As usual, we make the following mild assumption about the customer's cost/disutility function:
\vspace{-0.05in}
\begin{assumption}\label{as:CUSTconvex}
	$\forall i, \forall t$, $C_i(\cdot; t)$ is convex and differentiable with $C_i(0; t)=0$ and $C'_i(0; t)=0$.
\end{assumption}
\vspace{-0.05in}

Under this assumption, if the customers are left to decide their power consumption themselves without any demand response program, each customer will choose to consume at their power demand $d_i(t)$, i.e., $x_i(t)=0$.
The convexity assumption is consistent with the concavity assumption of customer utility functions as was done in \cite{li2011optimal,jiang11,bitar12,zhao2014optimal}.


In reality, this cost function may not be known until at (or just before) the time of consumption.
In some cases, the LSE needs to estimate the customers cost functions to set the appropriate price for demand response payment.
Here, we assume that Assumption \ref{as:CUSTconvex} also applies to the estimated customer cost function $\hat{C}_i(x_i; t)$.

A simple but widely used example is the quadratic function, i.e., $C_i(x_i(t); t)=a_i(t) x_i(t)^2$ which is explored further in Section \ref{sec:char_quad}. 
The uncertainty of the function is therefore represented by the randomness in its parameter $a_i(t)$.
While simple, quadratic cost functions are widely used in electricity market literature~\cite{allaz1993cournot, yao2007two, murphy2010impact, cai2013inefficiency, liu2014pricing}.


\vspace{0.05in}
\noindent\textbf{Load Serving Entity}

We consider the general case where the LSE has volatile renewable energy generation that must be used when it is produced.
The generation at time $t$ is denoted by a random variable $r(t)$ bounded between $\underline{r}$ and $\overline{r}$.

The LSE procures power beforehand according to its estimation on customers' demand $\sum_i\hat{d}_i(t)$ and renewable generations $\hat{r}(t)$ for timeslot $t$.
We assume the procurement is $\sum_i\hat{d}_i(t)-\hat{r}(t)$, while our model and approaches can be easily extended to handle other form of procurement.

After the power procurement, LSE is responsible for balancing in real time the power demand, which is the aggregate customer power consumption $\sum_i d_i(t)$ and the power supply, which consists of both the power procured beforehand, e.g., $\sum_i\hat{d}_i(t)-\hat{r}(t)$ plus the available renewable power $r(t)$.

Therefore, the demand response goal is to clear the mismatch due to prediction errors, i.e.,
\begin{align}
	D(t) & =\sum_i d_i(t)-\sum_i(\hat{d}_i(t)-\hat{r}(t)+r(t))\nonumber \\
	 & =\sum_i (d_i(t)-\hat{d}_i(t))-(r(t)-\hat{r}(t)).\nonumber
\end{align}

\vspace{-0.05in}

We denote by $\delta_i(t) := d_i(t)-\hat{d}_i(t)$ and $\delta_r(t) := r(t)-\hat{r}(t)$ the prediction errors for customer $i$'s demand and the renewable generation, respectively.

Then
\vspace{-0.05in}
\begin{align}
	D(t)=\sum_i \delta_i(t) - \delta_r(t). \nonumber
\end{align}
\vspace{-0.05in}

In many cases, the LSE cannot get exactly $D(t)$ amount of demand response, and has to bear the cost denoted by the penalty function $C_g(\Delta(t))$, where $\Delta(t)=D(t)-\sum_i x_i(t)$.
Specifically, this is the cost imposed on the LSE to close the gap through actions such as employing fast responding reserves or grid energy storage.
We note that the cost of reducing some of the mismatch via hour-ahead (or other near-real-time) power market interactions can be incorporated in $C_g(\Delta(t))$. 
Again, we make the following mild assumption:


\vspace{-0.05in}
\begin{assumption}\label{as:LSEconvex}
	$C_g(\cdot)$ is convex and differentiable with $C_g(0)=0$ and $C'_g(0)=0$.
\end{assumption}
\vspace{-0.05in}

This convexity assumption for the LSE's cost was also made \cite{li2011optimal,jiang11}.

In order for the LSE to tolerate the mismatch and prevent blackouts, the LSE must purchase long-term energy storage or reserves in some forward market denoted by $\kappa$. 
This gives us the mismatch constraint\footnote{This constraint can be generalized to the case where the upper and lower bounds are function of $\kappa$, and our approaches apply with little change.}:


\vspace{-0.2in}
\begin{align}
	-\kappa \leq \Delta(t) \leq \kappa, \forall t. \label{const:modelCAP}
\end{align}
Notice that LSE's decision for the capacity $\kappa$ is constant for all $t\in\{1,2,...,T\}$
\fixes{where the time horizon $T$ is can be set to the time for deciding the capacity, e.g. a day, month, etc.}
Denote $C_{\text{cap}}(\kappa)$ as the cost of the energy storage/reserve capacity $\kappa$ amortized to the duration of interests $T$.
Again, we make the following mild assumption:


\vspace{-0.05in}
\begin{assumption}\label{as:kappaconvex}
	$C_{\text{cap}}(\cdot)$ is increasing, convex and differentiable.
\end{assumption}
\vspace{-0.05in}

For instance, $C_g(\cdot)$ can be a quadratic function $A\Delta(t)^2$, and $C_k(\kappa)$ can be linear $c\kappa$. While simple, quadratic cost functions are widely used in generation cost modeling~\cite{park1993economic,gaing2003particle,hug2012generation}.

\noindent\textbf{Optimization Problem}

\fixes{The expected social cost can be represented by 
\begin{align}
	C_{\text{cap}}(\kappa) + \mathbb{E}_{\boldsymbol{\diffd},\delta_r,C_i(\cdot)}\left[\sum_t\left[\sum_{i}C_i(x_i(t); t)+C_\text{g}\left(D(t)-\sum_{i}x_i(t)\right)\right]\right].\nonumber
\end{align}
where we assume that the randomness in customers' cost function $C_i(\cdot;t)$ and the mismatch $D(t)$ are stationary.
This assumption is reasonable since the randomness in $D(t)$ is due to the \emph{prediction error} of the customers' load demands and renewable energy supply.
Additionally, this assumption is intuitive for customers whose underlying load preference behavior does not change significantly within the time horizon $T$.
Therefore, we can remove the time dependencies and simplify the expected social cost to:}
\removed{The social cost can be represented by 
\begin{align}
	C_{\text{cap}}(\kappa) + \sum_t\left[\sum_{i}C_i(x_i(t); t)+C_\text{g}\left(D(t)-\sum_{i}x_i(t)\right)\right].\nonumber
\end{align}
Note this function is \emph{random} due to the randomness in customers' cost function $C_i(\cdot;t)$ and the mismatch $D(t)$ where we assume that both sources are stationary.
This assumption is reasonable since the randomness in $D(t)$ is due to the \emph{prediction error} of the customers' load demands and renewable energy supply.
Additionally this assumption is intuitive for customers whose underlying load preference behavior does not change significantly within the time horizon $T$.}

\removed{In practice, as the duration $T$ is fairly large, we can use the expectation and the law of large numbers to approximate the social cost:}

\vspace{-0.2in}
\begin{align}
	C_{\text{cap}}(\kappa) + \mathbb{E}_{\boldsymbol{\diffd},\delta_r,C_i(\cdot)}\left[\sum_{i}C_i(x_i)+C_\text{g}\left(D-\sum_{i}x_i\right)\right] \label{eq:expected-cost}
\end{align}
where $C_{\text{cap}}(\kappa)$ is amortized to a single timeslot.

The goal is therefore to decide the capacity planning $\kappa$ and a practical policy $\mathbf{x}(\boldsymbol{\diffd},\delta_r)$ \emph{simultaneously} to optimize the \emph{expected} social cost \eqref{eq:expected-cost}.
\begin{subequations}\label{opt:general}
	\begin{align}
		\min_{\kappa,\mathbf{x}(\boldsymbol{\diffd},\delta_r)} & C_{\text{cap}}(\kappa)\nonumber \\
		&+\mathbb{E}_{\boldsymbol{\diffd},\delta_r,C_i(\cdot)}\left[\sum_{i}C_i(x_i(\delta_i,\delta_r))+C_\text{g}\left(D-\sum_{i}x_i(\delta_i,\delta_r)\right)\right] \nonumber \\
		\text{s.t. } & \max_{\boldsymbol{\diffd},\delta_r}\left\{D-\sum_{i}x_i(\delta_i,\delta_r)\right\}\leq \kappa \label{const:CAPpolicy1} \\
		&\min_{\boldsymbol{\diffd},\delta_r}\left\{D-\sum_{i}x_i(\delta_i,\delta_r)\right\}\geq -\kappa. \label{const:CAPpolicy2}
	\end{align}
\end{subequations}

We note that \eqref{const:CAPpolicy1} and \eqref{const:CAPpolicy2} are worst-case constraints. 
To optimize the policy is known to be challenging, so we first provide the upper and lower bound to this optimization.

\noindent\textbf{Upper bound: SEQ}

The upper bound is the two-stage policy used widely in practice which first obtains capacity assuming the worst-case mismatch due to prediction error and then sets a price on voluntary DR. 
We call it ``SEQ'' in this paper to highlight it optimizes sequentially instead of simultaneously.
In the first stage, SEQ performs capacity planning to obtain $\kappa$ by solving the following optimization problem:

\vspace{-0.1in}
\begin{align}
	\min_{\kappa}\text{ } & C_\text{cap}(\kappa) \label{opt:seq-1} \\
	\text{s.t. } & \max_{\boldsymbol{\diffd},\delta_r} \left\{D\right\}\leq \kappa \nonumber \\
	& \min_{\boldsymbol{\diffd},\delta_r} \left\{D\right\}\geq -\kappa \nonumber
\end{align}

Then in real-time, SEQ sets price to extract demand response from customers.
This is to mimic the voluntary demand response programs such as NYISO EDRP. 

Formally, SEQ works as:

\vspace{0.05in}
\noindent
\fbox{\begin{minipage}{26em}
\textbf{SEQ:}
\begin{itemize}[leftmargin=0.25in]
\item Solve \eqref{opt:seq-1} to get $\kappa^{SEQ}$ for capacity planning;
\item In real time, pick a function $p(D; \kappa^{SEQ})$ to decide the demand response payment price $p$ when observing a mismatch $D$.
\end{itemize}
\end{minipage}}
\vspace{0.05in}

Note how to pick the function highly depends on the experience and expertise of the LSE staffs. In Section~\ref{sec:alg}, we propose a data-driven approach to obtain $p(D)$.

\noindent\textbf{Lower bound: OFFLINE}

A lower bound on the minimum social cost is given by the \emph{offline/a-posteriori} optimal solution. Specifically, the offline optimum is given by the following:
\begin{subequations}\label{opt:OFFLINE1}
	\begin{align}
		\min_{\kappa}\text{ } & C_\text{cap}(\kappa) + \mathbb{E}_{\boldsymbol{\delta},\delta_r}\min_{\mathbf{x}(t)}\left\{\sum_{i}C_i(x_i(t);t)+C_\text{g}\left(D(t)-\sum_{i}x_i(t)\right)\right\} \nonumber \\
		\text{s.t. } & -\kappa\leq D(t)-\sum_{i}x_i(t)\leq\kappa,\quad \text{for each realization}~t. \label{const:Capacity}
	\end{align}
\end{subequations}

Note that, the minimization over $\mathbf{x}(t)$ is performed \emph{inside} the expectation, meaning that it is performed after observing the realizations of the random variables. This results in the fact that the offline optimum can never be beaten by any online policy $\mathbf{x}(\delta_i, \delta_r)$. 

Formally, OPT works as:

\vspace{0.05in}
\noindent
\fbox{\begin{minipage}{26em}
\textbf{OPT:}
\begin{itemize}[leftmargin=0.25in]
\item Solve \eqref{opt:OFFLINE1} to get $\kappa^{*}$ for capacity planning;
\item In real time, solve the inner minimization over $\mathbf{x}(t)$ in \eqref{opt:OFFLINE1} to get $\mathbf{x}^*(t)$.
\end{itemize}
\end{minipage}}
\vspace{0.05in}

\section{Characterizing the optima}
\label{sec:char}

In this section, we provide the characterization of the optimal solution to reveal special structures that we take advantage of in our algorithm design (Section \ref{sec:alg}).
We start with the convexity of the problem followed by a concrete case study of the necessary and sufficient conditions of the optimal solution.


\vspace{-0.05in}
\subsection{Convexity}

The first key result regarding the problem is the convexity, as stated formally in Theorem~\ref{th:offline-convexity}.
The convexity is crucial for our proposed algorithm in Section~\ref{sec:alg}.

\begin{theorem}
\label{th:offline-convexity}
	\eqref{opt:OFFLINE1} is a convex optimization problem over $\kappa$. 
\end{theorem}
\begin{proof}
	The result follows from Assumption \ref{as:kappaconvex} and Lemma \ref{lm:convexKAPPA} given below.
\end{proof}
The proof requires the following lemmas and we restate the real-time decision (i.e. inside the expectation) of Problem \eqref{opt:OFFLINE1} as:
\begin{subequations}\label{opt:OFFLINE1_REALTIME}
	\begin{align}
		R(\kappa;t):=\min_{\mathbf{x}(t)}\text{ } & \left\{\sum_{i}C_i(x_i(t);t)+C_\text{g}\left(D(t)-\sum_{i}x_i(t)\right)\right\} \\
		\text{s.t. } & -\kappa\leq D(t)-\sum_{i}x_i(t)\leq\kappa. \label{const:Capacity_RT}
	\end{align}
\end{subequations}

\begin{lemma}\label{thm:convexity2}
	Problem \eqref{opt:OFFLINE1_REALTIME} is a convex optimization problem.
\end{lemma}

\begin{proof}
	This comes straightforward from Assumptions \ref{as:CUSTconvex} and \ref{as:LSEconvex}, that the function inside $C_{g}(\cdot)$ is linear, and the constraints form a convex set.
\end{proof}

\begin{lemma}\label{lm:convexKAPPA}
	$R(\kappa;t)$ as defined by \eqref{opt:OFFLINE1_REALTIME} is a convex function of $\kappa$. Additionally the negative of the sum of dual variables $\underline{\theta}+\overline{\theta}$ from constraint \eqref{const:Capacity_RT} 
	is the subgradient of $R(\kappa;t)$ w.r.t. $\kappa$.
\end{lemma}
\begin{proof}
	Since \eqref{opt:OFFLINE1_REALTIME} is a convex optimization problem (from Lemma \ref{thm:convexity2}), the result follows from the fact that Problem \eqref{opt:OFFLINE1_REALTIME} is in a perturbed form in terms of $\kappa$ which can be used to follow standard sensitivity analysis (See Section 5.6.1 in \cite{boydconvex}).
\end{proof}


\subsection{Case study: quadratic cost functions}
\label{sec:char_quad}

To provide more optimal structures, we restrict ourselves to the quadratic cost functions, i.e. $C_i(x)=a_ix^2_i$ and $C_g(\mathbf{x})=A(D-\sum_i x_i)^2$.
Their key advantageous property is that have linear marginal costs which allow for more concrete analysis and motivate the linear decision policy (Section \ref{sec:alg_lin}).
Note this may seem restrictive, but this form is standard within the electricity markets literature, e.g., \cite{allaz1993cournot, yao2007two, murphy2010impact, cai2013inefficiency, liu2014pricing} and generation cost modeling \cite{park1993economic,gaing2003particle,hug2012generation}.


Then the offline optimization problem in real-time without capacity constraints\footnote{This can be added with additional presentation complexity.} becomes: 
\vspace{-0.05in}
\begin{align}
	\min_{\mathbf{x}}\left\{\sum_{i}a_ix_i^2+A\left(D-\sum_{i}x_i\right)^2\right\}\nonumber
\end{align}

By the first-order stationary conditions from the Karush-Kuhn-Tucker (KKT) conditions (See Appendix \ref{sec:KKT} for the general case), we have
\vspace{-0.05in}
\begin{align}
	a_ix_i=A(D-\sum_jx_j)\text{, so } x_i=\frac{A}{a_i}(D-\sum_jx_j)\nonumber
\end{align}

Sum over all $i$ gives:
\begin{align}
	\sum_ix_i=\left(\sum_i\frac{A}{a_i}\right)(D-\sum_ix_i)\text{, so } \sum_ix_i=\frac{\sum_i\frac{A}{a_i}}{1+\sum_i\frac{A}{a_i}}D\nonumber
\end{align}

Finally, we have $x^*_i=\frac{A}{\left(1+\sum_j\frac{A}{a_j}\right)a_i}D$, which is linear to the social mismatch $D$. This property is crucial to help us design the linear policy.
This approach is general and can be applied to more complicated cost functions. 
The difference is that the optimal structure would be more complicated, and so would the policy.

\section{Algorithms}
\label{sec:alg}

%

In this section, we design two algorithms/control policies for DR, namely, ``PRED'' and ``LIN''.
PRED is the prediction based policy with no contracts, and is hence a voluntary program.
LIN relies on a linear demand response contract between LSE and customers, and is hence a mandatory program: it, however, respects customer uncertainties as will be shown later in this section. 
For presentation simplicity, we omit the $t$ and will highlight it wherever needed.

\vspace{-0.05in}
\subsection{Prediction Based Policy}
Ideally, the LSE would like to set a price for demand response and obtain an amount that would optimally solve Problem \eqref{opt:OFFLINE1_REALTIME}.  This in turn would allow Problem \eqref{opt:OFFLINE1} of finding the optimal capacity to be solved.

In general however, the customer cost functions $C_i(x_i(t);t)$ and the instantaneous customer load deviations $\delta_i(t)$ are not known by the LSE.
\fixes{Despite this, from historical data of past customer responses to prices, the LSE can estimate their cost functions denoted by $\hat{C}_i(x_i(t);t)$ with some inaccuracies.}
\removed{Despite this, from historical data of past customer responses to prices, the LSE may be able to estimate their cost functions by $\hat{C}_i(x_i(t);t)$.
Load data can also be used to form probability density functions for each customer's load deviation.}

\fixes{From a privacy perspective, customers may not be willing to share their detailed load data due to the information that can be inferred from it.
However, customers can hide some of their private information by recently proposed mechanism, e.g., using batteries to add randomness to their consumptions~\cite{zhao2014achieving,laforet2016towards}.}

First, the LSE utilizes the historical data to solve a (deterministic) optimization problem about how to set the price for demand response, i.e., $p(D;\kappa)$ for a given $\kappa$.  For a given $\kappa$ 
it optimizes:
\begin{align}
	H(\kappa;D):=\min_{p}\text{ } & \left\{\sum_{i}\hat{C}_i(x_i(p))+C_\text{g}\left(D-\sum_{i}x_i(p)\right)\right\} \label{opt:PRED_KAPPA_D} \\
	\text{s.t. } & -\kappa\leq D-\sum_{i}x_i(p)\leq\kappa. \nonumber
\end{align}
We abuse the notation of $x_i(\cdot)$ as $x_i(p)=\arg\min\hat{C}_i(x_i)-px_i$ to represent a customer's demand response given a particular price $p$.  

To find the optimal amount of capacity $\kappa$, it must solve
\begin{align}
	\min_{\kappa}C_\text{cap}(\kappa)+\mathbb{E}_{\boldsymbol{\diffd},\delta_r}\left[H(\kappa;D)\right]\label{opt:PRED_KAPPA}
\end{align}
which may be done by exhaustive search for the best $\kappa$.  Once $\kappa$ is fixed, then the optimal solution to \eqref{opt:PRED_KAPPA_D} gives a price function $p(D;\kappa)$. 
This price function is then applied in real-time to obtain demand response from all the customers.
Under certain conditions, such as quadratic and/or linear cost functions, a closed form of $p(D;\kappa)$ can be obtained.

Formally, PRED works as follows:

\vspace{0.05in}
\noindent
\fbox{\begin{minipage}{26em}
\textbf{PRED:}
\begin{itemize}[leftmargin=0.25in]
\item Using historical date to solve \eqref{opt:PRED_KAPPA_D} to get $p(D; \kappa)$;
\item Solve \eqref{opt:PRED_KAPPA} to get $\kappa^{PRED}$ for capacity planning;
\item In real time, when observing $D$, set price $p(D; \kappa^{PRED})$.
\end{itemize}
\end{minipage}}
\vspace{0.05in}


PRED is intuitive and works well when the uncertainties on $C_i(\cdot)$ are small, as illustrated in the numerical evaluations in Section~\ref{sec:linear-evaluation}. 
This makes PRED an attractive solution in many cases.
For instance, in the emergency demand response program, the LSE sets the price based on their knowledge and experience, which corresponds to the predicted $p(D;\kappa)$ in our PRED scheme.
We would like to highlight that by proposing PRED, we are making two improvements over existing schemes.
First, we provide a rigorous procedure to get $p(D;\kappa)$ by looking into historical data and solving an optimization problem. 
Second, we jointly solve $\kappa$ and $p(D;\kappa)$, instead of getting them sequentially in SEQ.

However, its performance greatly depends on the estimation accuracy of the cost functions, and degrades as the uncertainties increase.
To see this, note that the LSE needs to pick $p(D;\kappa)$ based on some sort of ``averaged'' customers' cost functions, without knowing the exact realization of $C_{i}(\cdot;t)$.
When $C_{i}(\cdot;t)$ has more uncertainties, the response $x_i(p,C_i(\cdot;t))$ from customers based on the realized $C_{i}(\cdot;t)$ may be farther away from the value expected by the LSE.

\vspace{-0.05in}
\subsection{Linear policy}
\label{sec:alg_lin}
Motivated by the potential inefficiency of PRED when uncertainties are large, we design a linear policy, ``LIN'', as the contract between LSE and customers. 

In general, we can take the following nonlinear form of a customer demand response policy: 
\begin{align}
	x_i({\diffd}_i,D)=\alpha_if_i(D)+\beta_ig_i({\diffd}_i)+\gamma_i
\end{align}
where $f_i(D)$ is a function of net aggregate mismatch $D$, 
and $g_i({\diffd}_i)$ is function of the local demand mismatch. This structure is rich enough for the demand change to be responsive to the net society mismatch through $\alpha_if_i(D)$, local self mismatch through $\beta_ig_i({\diffd}_i)$, and an independent reduction of $\gamma_i$.

While the above form is general and powerful, there are two concerns around it. First of all, in addition to the parameter $\alpha_i, \beta_i, \gamma_i$, we need to optimize over the function form $f_i(\cdot)$ and $g_i(\cdot)$, which is very challenging, if not impossible.
More importantly, customers may not be willing to accept contracts in a very complicated form.
Based on these considerations, we restrict our attentions to the following linear policy, and show that this policy works very well in many cases, in particular, when cost functions are quadratic.
How to decide the optimal $f_i(\cdot)$ and $g_i(\cdot)$ for general convex cost functions is our future work.

We now focus on a simple but powerful linear demand response policy that is a function of total and local net demands:
\begin{align}
	x_i({\diffd}_i)=\alpha_iD+\beta_i{\diffd}_i+\gamma_i. \label{pol:GENlinear}
\end{align}

Intuitively, there are three components: $\alpha_iD$ implies each customer shares some (predefined) fraction of the global mismatch $D$; $\beta_i{\diffd}_i$ means  customer $i$ may need to take additional responsibility for the mismatch due to his own demand fluctuation and estimation error; finally, $\gamma_i$, the constant part, can help when the random variables $\mathbb{E}[D]$ and/or $\mathbb{E}[\delta_i]$ is nonzero.
\fixes{In fact, this linear policy is optimal for the case of quadratic cost functions which can be seen in Section \ref{sec:char_quad}.}

%

Then the LSE needs to solve the following optimization problem to obtain the optimal parameters for the linear contract, i.e., $\boldsymbol{\alpha}, \boldsymbol{\beta}, \boldsymbol{\gamma}$, as well as the optimal capacity $\kappa$:
\begin{align}
	\min_{\boldsymbol{\alpha},\boldsymbol{\beta},\boldsymbol{\gamma},\kappa}\text{ } &
	C_\text{cap}\left(\kappa\right)+\sum_{i}\mathbb{E}_{\delta_i, \delta_r, C_i}\left[C_i(\alpha_iD+\beta_i{\diffd}_i+\gamma_i)\right]\nonumber \\
	 & \quad +\mathbb{E}_{\boldsymbol{\diffd}, \delta_r}\left[C_\text{g}\left(D-\sum_{i}(\alpha_iD+\beta_i{\diffd}_i+\gamma_i)\right)\right] \label{opt:GENall2} \\
	\text{s.t. } & \eqref{const:CAPpolicy1}, \eqref{const:CAPpolicy2}\nonumber
\end{align}


The expectation can be separated by customer and LSE because of the linearity of expectation property.


\begin{theorem}
	Problem \eqref{opt:GENall2} is a convex optimization problem.
\end{theorem}
\begin{proof}
	The objective function is convex in $(\boldsymbol{\alpha},\boldsymbol{\beta},\boldsymbol{\gamma})$ because the expectation operator preserves convexity, the cost functions are convex from Assumptions \ref{as:CUSTconvex}-\ref{as:kappaconvex}, and the arguments inside each cost function are linear in $(\boldsymbol{\alpha},\boldsymbol{\beta},\boldsymbol{\gamma})$ for each realization of $(\delta_i,\delta_r)$ which preserve convexity.
	By substituting the linear policy \eqref{pol:GENlinear} into constraints \eqref{const:CAPpolicy1}, \eqref{const:CAPpolicy2} it forms a convex set because the point-wise maximum (minimum) operator preserves convexity (concavity), and the arguments inside the operators are linear in $(\boldsymbol{\alpha},\boldsymbol{\beta},\boldsymbol{\gamma})$ for each realization of $(\delta_i,\delta_r)$ which preserve convexity (concavity).  The result follows from the objective function being convex and the constraint set forming a convex set.
\end{proof}


Since Problem \eqref{opt:GENall2} is a convex optimization problem and is also stochastic, it can be solved centrally with standard stochastic optimization techniques such as the Stochastic Subgradient Method with Monte Carlo sampling  where one simulates a sample of $(\delta_i,\delta_r)$ and solves the problem via the dual variables by taking a step in the direction of the dual variables' subgradient . As the number of iterations approaches infinity, then the dual variables approach their optimal values \cite{boyd2014stochsubgradient}.  If the cost functions are quadratic and/or linear then only the first and second moments of $\delta_i$ and $\delta_r$ are needed to make \eqref{opt:GENall2} into a deterministic convex optimization problem.


Formally, LIN works as:

\vspace{0.05in}
\noindent
\fbox{\begin{minipage}{26em}
\textbf{LIN:}
\begin{itemize}[leftmargin=0.25in]
\item Solve \eqref{opt:GENall2} to get $\kappa^{LIN}$ for capacity planning and the parameters $(\boldsymbol{\alpha}^*,\boldsymbol{\beta}^*,\boldsymbol{\gamma}^*)$;
\item In real time, customer $i$ needs to provide demand response $x_i(t)=\alpha_i^*D(t)+\beta_i^*{\diffd}_i(t)+\gamma_i^*$ when observing $D(t)$.
\end{itemize}
\end{minipage}}
\vspace{0.05in}

\subsection{Distributed algorithm design}

\revised{In many cases, the LSE may not know the information about customers' cost functions, and cannot solve \eqref{opt:GENall2} for the parameters in the linear contract.
To handle this, we design the distributed algorithm to decompose the problem, so that each customer solves an optimization problem based on her own cost function, and LSE solves another optimization problem without knowing the customers' cost function.
The goal is to achieve the optimal $(\kappa^{LIN}, \boldsymbol{\alpha}^*,\boldsymbol{\beta}^*,\boldsymbol{\gamma}^*)$.}
\removed{In many cases, the LSE may not know the information about customers' cost functions, and cannot solve \eqref{opt:GENall2} for capacity planning and the parameters in the linear contract.
To handle this, we design the distributed algorithm to decompose the problem, so that each customer is solving an optimization problem based on her own cost function, and LSE is solving another optimization problem without knowing the customers' cost function.
The goal is to achieve the optimal $(\kappa^{LIN}, \boldsymbol{\alpha}^*,\boldsymbol{\beta}^*,\boldsymbol{\gamma}^*)$ that can be used in LIN.}

\revised{To decompose Problem \eqref{opt:GENall2} we use the auxiliary variables $(u_i,v_i,w_i)$ for $(\alpha_i,\beta_i,\gamma_i)$ in the estimated cost function $\hat{C}_i(\cdot)$ of each customer, and add the constraints that $u_i=\alpha_i$, $v_i=\beta_i$, and $w_i=\gamma_i$:}
\removed{To decompose Problem \eqref{opt:GENall2} we introduce and substitute $(u_i,v_i,w_i)$ for $(\alpha_i,\beta_i,\gamma_i)$ in each of the customer's estimated cost function $\hat{C}_i(\cdot)$, and add the constraints that $u_i=\alpha_i$, $v_i=\beta_i$, and $w_i=\gamma_i$.
The equivalent problem of \eqref{opt:GENall2} becomes:}
\begin{subequations}\label{opt:GENall2D}
	\begin{align}
		\min_{\boldsymbol{\alpha},\boldsymbol{\beta},\boldsymbol{\gamma},\mathbf{u},\mathbf{v},\mathbf{w},\kappa}\text{ } &
		C_\text{cap}\left(\kappa\right)+\sum_{i}\mathbb{E}_{\boldsymbol{\diffd},\delta_r}\left[\hat{C}_i(u_iD+v_i{\diffd}_i+w_i)\right]\nonumber \\
		& +\mathbb{E}_{\boldsymbol{\diffd},\delta_r}\left[C_\text{g}\left(\sum_{i}({\diffd}_i-\alpha_iD-\beta_i{\diffd}_i-\gamma_i)-\delta_r\right)\right] \nonumber\\
		\text{s.t. } & \eqref{const:CAPpolicy1}, \eqref{const:CAPpolicy2} \\
		 & u_i=\alpha_i,\quad v_i=\beta_i,\quad w_i=\gamma_i,\quad \forall i\in\mathcal{V} \label{const:SEPARATE}
	\end{align}
\end{subequations}

Problem \eqref{opt:GENall2D} can be split where each customer controls its own $(u_i,v_i,w_i)$ and the LSE controls $(\boldsymbol{\alpha},\boldsymbol{\beta},\boldsymbol{\gamma})$
from the dual decomposition of constraint \eqref{const:SEPARATE}.
\revised{Let $(\pi_i,\lambda_i,\mu_i)$ be the dual prices for each customer corresponding to constraint \eqref{const:SEPARATE} which result in the following meanings: $\mu_i$ is the price paid for the guaranteed reduction in power consumption $w_i$ by the customer, $\pi_i$ is the price paid for absorbing the fraction $u_i$ of the net aggregate system demand $D$, $\lambda_i$ is the price paid for absorbing the fraction $v_i$ of its own local net demand ${\diffd}_i$.
Therefore $\pi_iu_i+\lambda_iv_i+\mu_iw_i$ is the total payment to customer $i$ for following the linear demand response policy. }
\removed{Let $(\pi_i,\lambda_i,\mu_i)$ be the dual prices for each customer corresponding to constraint \eqref{const:SEPARATE}.
The dual price $\mu_i$ is the price paid for the guaranteed reduction in power consumption $w_i$ by the customer.
The dual price $\pi_i$ is the price paid for absorbing the fraction $u_i$ of the net aggregate system demand $D$.
The dual price $\lambda_i$ is the price paid for absorbing the fraction $v_i$ of its own local net demand ${\diffd}_i$.
Therefore $\pi_iu_i+\lambda_iv_i+\mu_iw_i$ is the total payment to customer $i$ for following the linear demand response policy. }

Accordingly, \eqref{opt:GENall2} is decomposed as the individual customer optimization problem,
\begin{align}
	\min_{u_i,v_i,w_i} \mathbb{E}_{\boldsymbol{\diffd},\diffd_r}\left[\hat{C}_i(u_iD+v_i{\diffd}_i+w_i)\right]-\pi_iu_i-\lambda_iv_i-\mu_iw_i \label{opt:CUSTalg2}
\end{align}
and the LSE's optimization problem,
\begin{align}
	\min_{\boldsymbol{\alpha},\boldsymbol{\beta},\boldsymbol{\gamma},\kappa}\text{ } & 	C_\text{cap}\left(\kappa\right)+\sum_{i}(\pi_i\alpha_i+\lambda_i\beta_i+\mu_i\gamma_i)\nonumber \\
	 & \quad  +\mathbb{E}_{\boldsymbol{\diffd},\diffd_r}\left[C_\text{g}\left(\sum_{i}({\diffd}_i-\alpha_iD-\beta_i{\diffd}_i-\gamma_i)-\diffd_r\right)\right] \label{opt:LSEalg2} \\
	 \text{s.t. } & \eqref{const:CAPpolicy1}, \eqref{const:CAPpolicy2}\nonumber
\end{align}
\revised{which both can be solved with standard stochastic optimization techniques such as the Stochastic Subgradient Method with Monte Carlo sampling \cite{boyd2014stochsubgradient}. With linear or quadratic cost functions, they in fact can be solved as a deterministic convex optimization problem with only the first and second order moments of $\boldsymbol{\diffd}$ and $\delta_r$.
The goal now becomes to find the optimal dual prices to ensure the customers' and LSE's decisions satisfy \eqref{const:SEPARATE}.
We achieve this by using the Subgradient Method (see \cite{bertsekas1999nonlinear} Chapter 6):}
\removed{Problems \eqref{opt:CUSTalg2} and \eqref{opt:LSEalg2} can be solved with standard stochastic optimization techniques such as the Stochastic Subgradient Method with Monte Carlo sampling \cite{boyd2014stochsubgradient}. With linear or quadratic cost functions, they in fact can be solved as a deterministic convex optimization problem whose parameters are determined by first and second order moments of $\boldsymbol{\diffd}$ and $\delta_r$.
To solve the decomposed problems, we must ensure the customers' and LSE's decisions satisfy \eqref{const:SEPARATE}. We achieve this by using the Subgradient Method (see \cite{bertsekas1999nonlinear} Chapter 6) to obtain the optimal dual prices in the following}

\vspace{0.05in}
\noindent
\fbox{\begin{minipage}{26em}
\textbf{Distributed Algorithm for LIN}:
\begin{enumerate}[start=0]
	\item \textbf{Initialization:} $(\boldsymbol{\alpha},\boldsymbol{\beta},\boldsymbol{\gamma},\mathbf{u},\mathbf{v},\mathbf{w},\boldsymbol{\pi},\boldsymbol{\lambda},\boldsymbol{\mu}):=\mathbf{0}$.
	\item \textbf{LSE:} receives $(u_i,v_i,w_i)$ from each customer $i\in\mathcal{V}$.
	\begin{itemize}
		\item Solves Problem \eqref{opt:LSEalg2}	and updates $(\boldsymbol{\alpha},\boldsymbol{\beta},\boldsymbol{\gamma})$ with the optimal solution.
		\item Updates the stepsize:
		\vspace{-0.05in}
		\begin{equation}\label{step:STEPupdate}
		\eta=\frac{\zeta/k}{||(\boldsymbol{\alpha},\boldsymbol{\beta},\boldsymbol{\gamma})-(\mathbf{u},\mathbf{v},\mathbf{w})||_2}
		\end{equation}
		where $\zeta$ is a small constant and $k$ is the iteration number.
		\item Updates the dual prices, $\forall i\in\mathcal{V}$:
	\end{itemize}
	\begin{small}
	\begin{equation}\label{step:DUALupdate}
	(\pi_i,\lambda_i,\mu_i):=(\pi_i,\lambda_i,\mu_i)+\eta\left((\alpha_i,\beta_i,\gamma_i)-(u_i,v_i,w_i)\right)
	\end{equation}
	\end{small}
	\begin{itemize}
		\item Sends $(\pi_i,\lambda_i,\mu_i)$ to the each customer respectively.
	\end{itemize}
	\item \textbf{Customer $i\in\mathcal{V}$:} receives $(\pi_i,\lambda_i,\mu_i)$ from LSE.
	\begin{itemize}
		\item Solves Problem \eqref{opt:CUSTalg2} and updates $(u_i,v_i,w_i)$ with optimal solution.
		\item Sends $(u_i,v_i,w_i)$ to the LSE.
	\end{itemize}
	\item Repeat Steps 1-2 until $||(\boldsymbol{\alpha},\boldsymbol{\beta},\boldsymbol{\gamma})-(\mathbf{u},\mathbf{v},\mathbf{w})||_2\leq \epsilon$ where $\epsilon$ is the tolerance on magnitude of the subgradient.
\end{enumerate}
\end{minipage}}
\vspace{0.05in}

Since the step size $\eta$ is non-summable, square summable, and diminishes to zero, the algorithm will converge to an equilibrium point.
Then the final agreed values are $\boldsymbol{\alpha} := \mathbf{u}, \boldsymbol{\beta} := \mathbf{v}, \boldsymbol{\gamma} := \mathbf{w}$ last sent by the customers, and the dual prices are $(\boldsymbol{\pi},\boldsymbol{\lambda},\boldsymbol{\mu})$ last sent by the LSE.
\revised{When the LSE signals the customers for DR and they respond accordingly, they will each be paid $\pi_iu_i+\lambda_iv_i+\mu_iw_i$.}
\removed{When the LSE signals the customers for DR and they respond accordingly, they will each be paid $\pi_iu_i+\lambda_iv_i+\mu_iw_i$ by the LSE.}

\revised{We now prove the convergence and optimality using two mild technical conditions:}
\removed{We now establish the convergence and optimality for the proposed distributed algorithm.
Two mild technical conditions are the following:}
\vspace{-0.05in}
\begin{assumption}\label{as:DUALbound}
	The distance between the first iteration of the distributed algorithm and the optimal dual prices has a finite upper bound:  $||\Lambda^{(1)}-\Lambda^*||_2\leq\overline{\Lambda}<\infty$.
\end{assumption}
\vspace{-0.10in}
\begin{assumption}\label{as:SUBGRADbound}
	The magnitude of the dual subgradient $G$ has a finite upper bound: $||G||_2\leq\overline{G}<\infty$ (i.e. the Lagrangian dual function is Lipschitz continuous).
\end{assumption}
\vspace{-0.05in}
\begin{theorem}\label{th:convergence}
	\revised{Given Assumptions \ref{as:DUALbound} and \ref{as:SUBGRADbound}, the distributed algorithm's trajectory of dual prices converge to the optimal dual prices.}
	\removed{Given Assumptions \ref{as:DUALbound} and \ref{as:SUBGRADbound}, the distributed algorithm's best dual prices converge to the optimal dual prices.}
\end{theorem}
The proof has been relegated to Appendix \ref{sec:app-1}.

\section{Performance Evaluation}
\label{sec:linear-evaluation}

\begin{figure}
	\includegraphics[width=0.85\columnwidth]{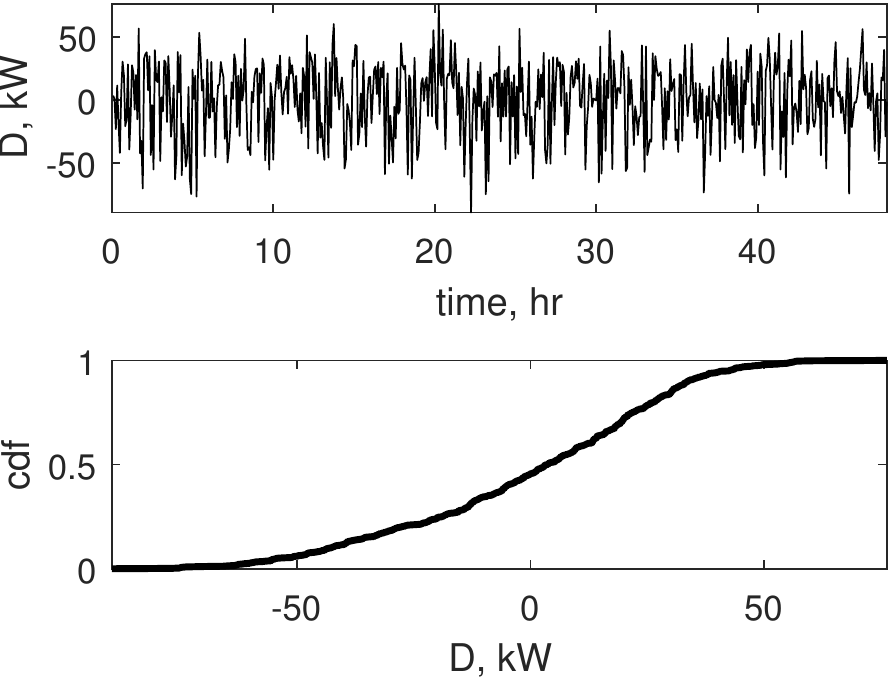}
	\vspace{-0.1in}
	\caption{(a) Demand mismatch for the 300 customers along with 100 kW of wind power capacity used in the simulations; (b) Cumulative distribution for the demand mismatch.}
	\label{f.D_trace}
	\vspace{-0.1in}
\end{figure}

\begin{figure*}[ht]
	\begin{center}
		\subfigure[Social Cost]{{\includegraphics[width=0.49\columnwidth]{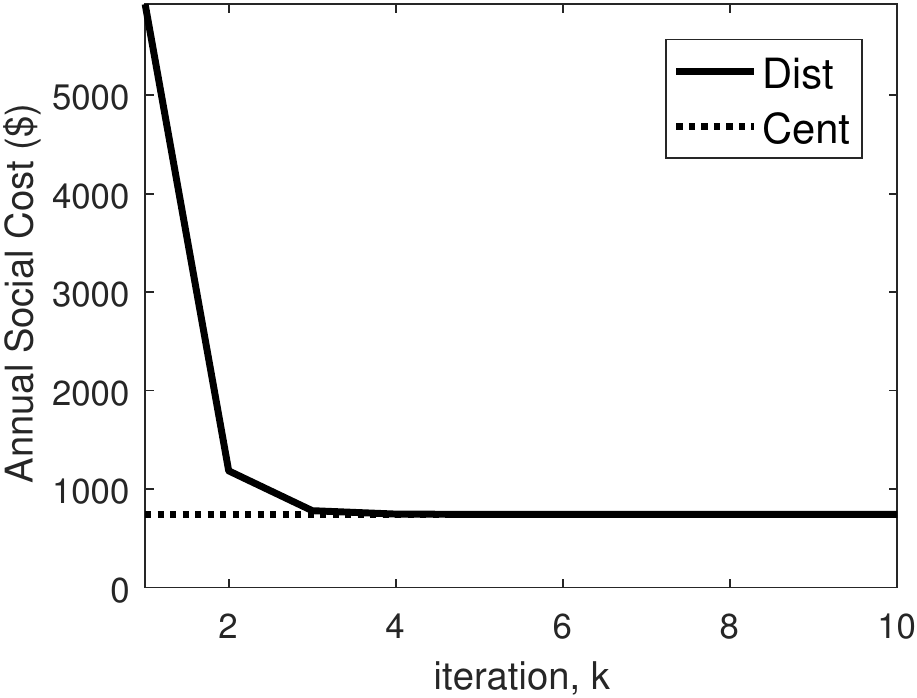}}
			\label{FIG:DISTALG_SOCCOST}}
		\subfigure[Total Reactive DR]{{\includegraphics[width=0.49\columnwidth]{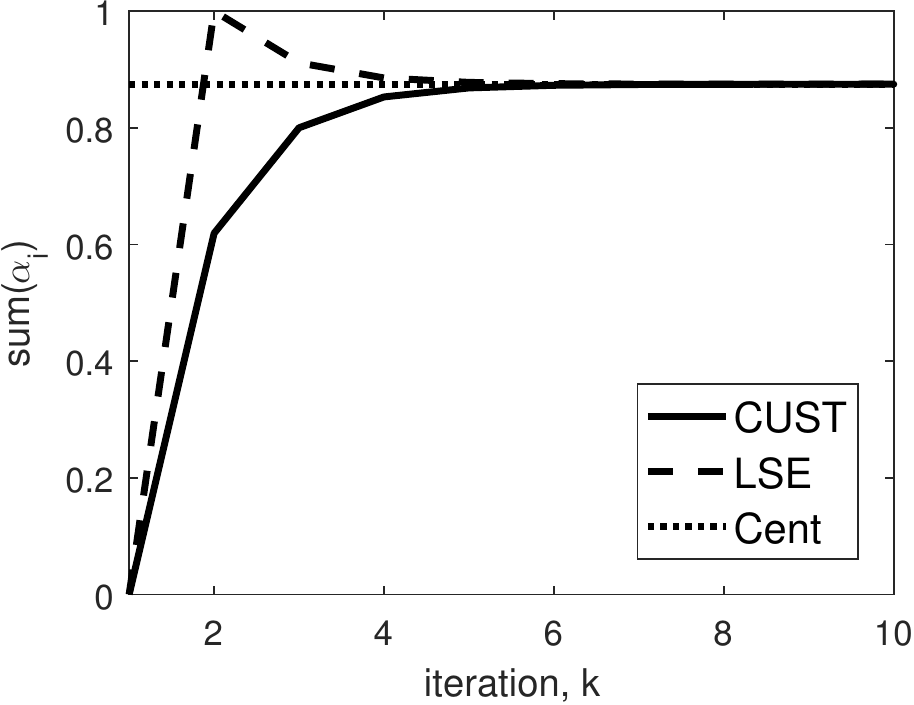}}
			\label{FIG:DISTALG_SUMALPHAS}}
		\subfigure[Total Constant DR]{{\includegraphics[width=0.49\columnwidth]{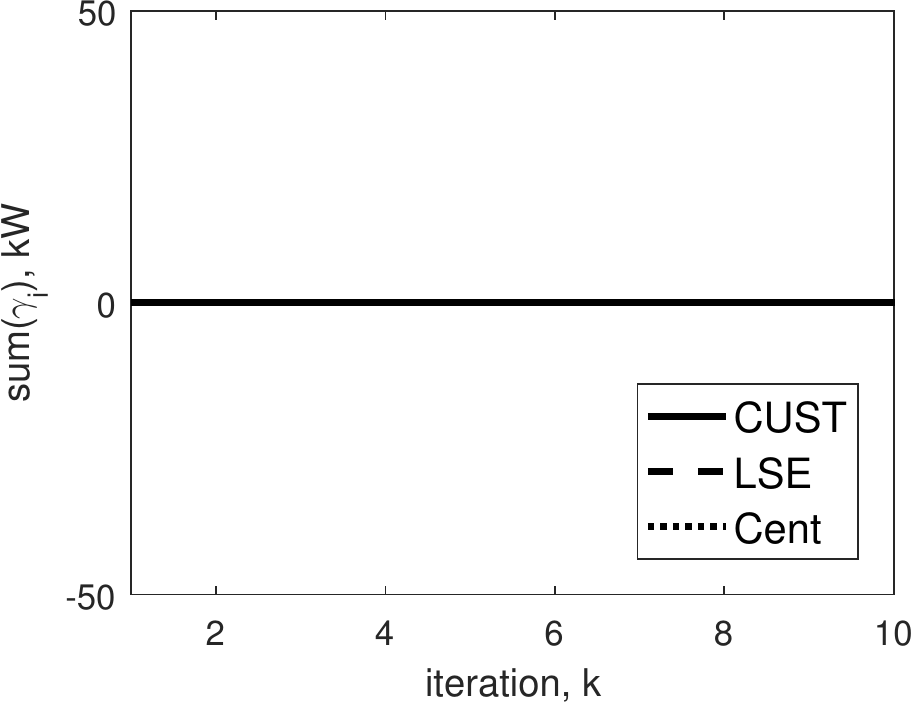}}
			\label{FIG:DISTALG_SUMGAMMAS}}
		\subfigure[Individual Customers]{{\includegraphics[width=0.49\columnwidth]{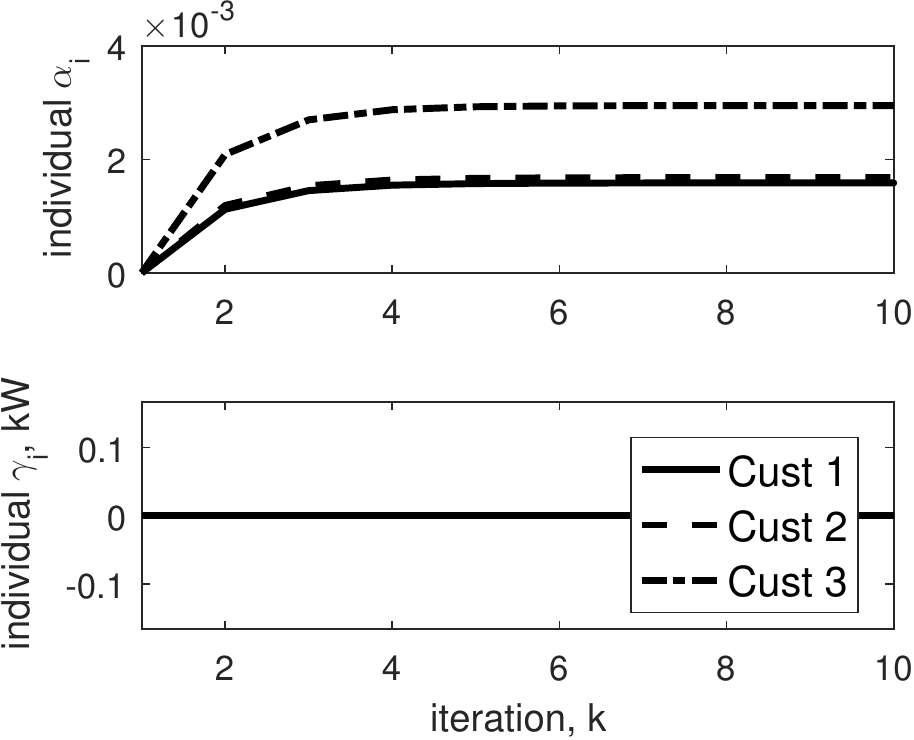}}
			\label{FIG:DISTALG_INDCUST}}
		\vspace{-0.15in}
		\caption{Convergence of the distributed algorithm.  The lines CUST and LSE in (b) represent the $\sum_iu_i$ and $\sum_i\alpha_i$ respectively as they converge to satisfy constraint \eqref{const:SEPARATE}.  Likewise for (c) w.r.t. $\sum_iw_i$ and $\sum_i\gamma_i$. (d) shows the convergence of three customers.}
		\label{f.DISTALG}
	\end{center}
	\vspace{-0.05in}
\end{figure*}

\begin{figure*}[!ht]
	\begin{center}
		\subfigure[Social Cost]{{\includegraphics[width=0.49\columnwidth]{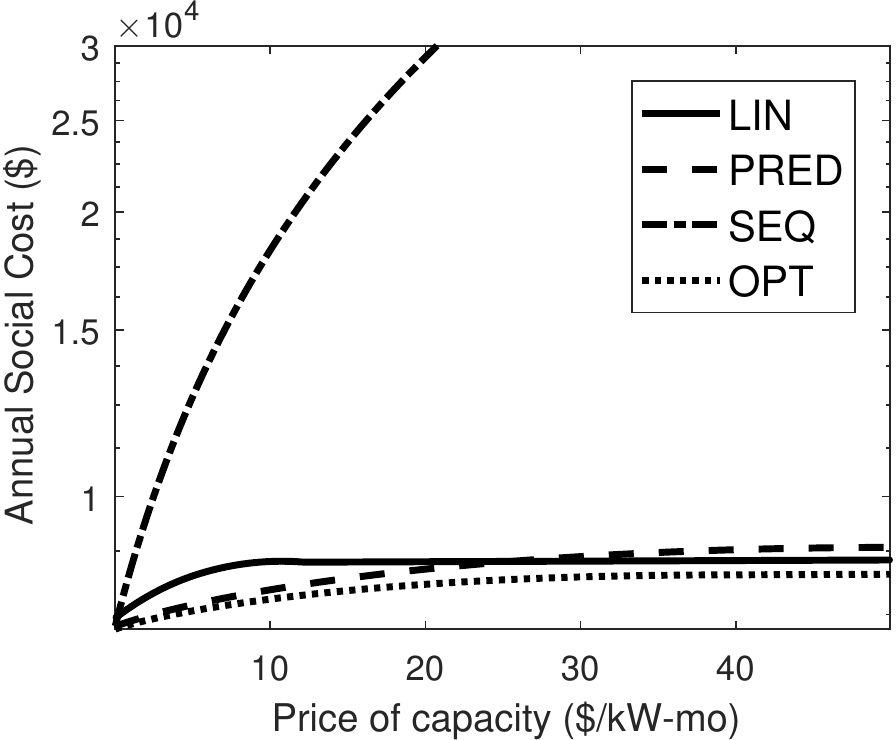}}
			\label{FIG:COMPSCH_SOC}}
				\subfigure[Optimal Capacity]{{\includegraphics[width=0.49\columnwidth]{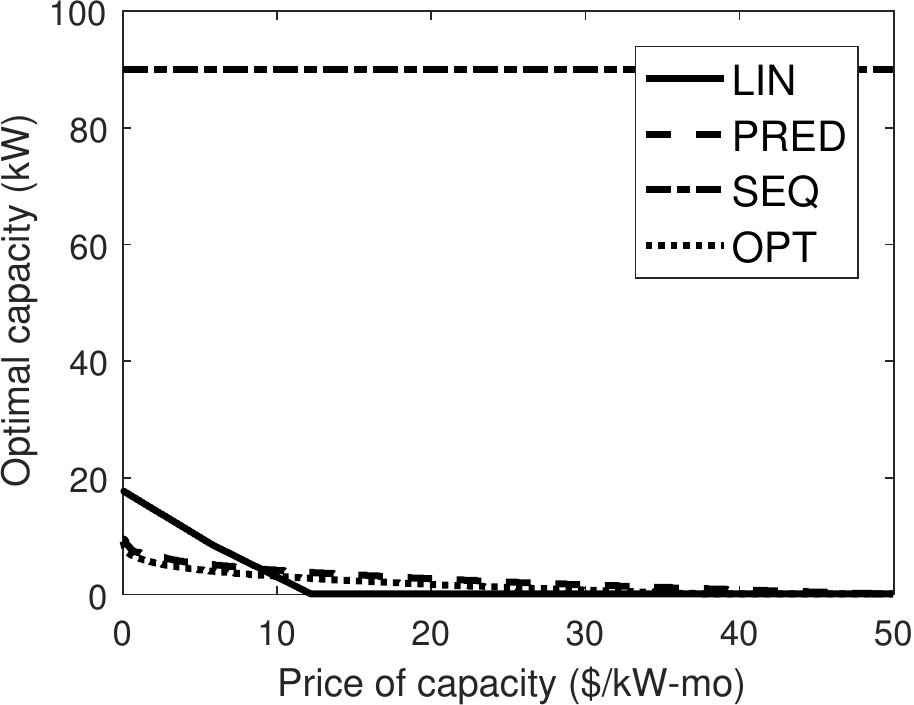}}
			\label{FIG:COMPSCH_OPTCAP}}
		\subfigure[Demand Response]{{\includegraphics[width=0.49\columnwidth]{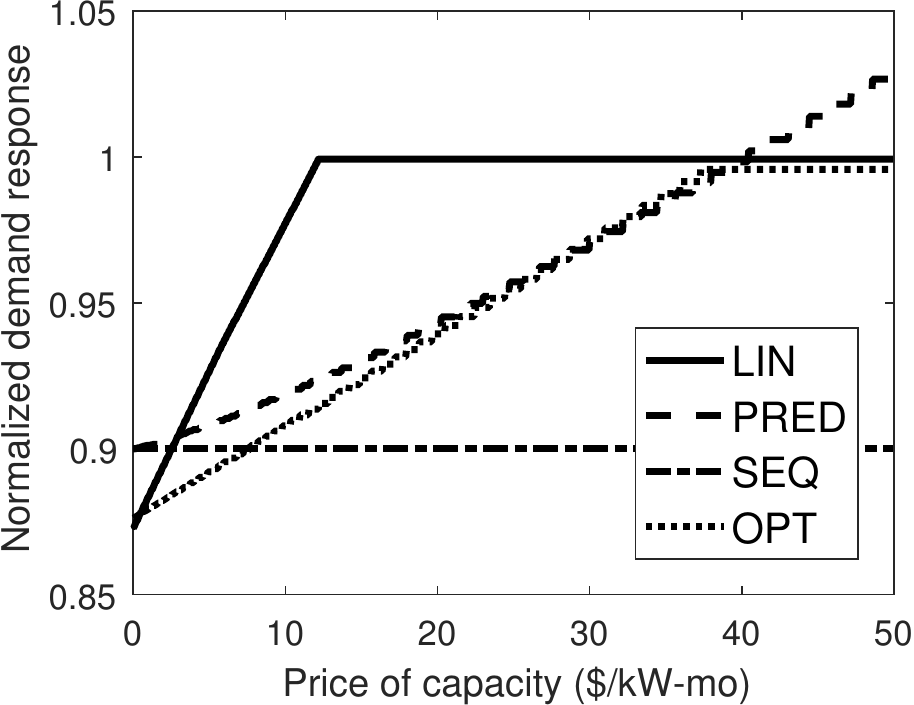}}
			\label{FIG:COMPSCH_DR}}
		\subfigure[Leftover Mismatch]{{\includegraphics[width=0.49\columnwidth]{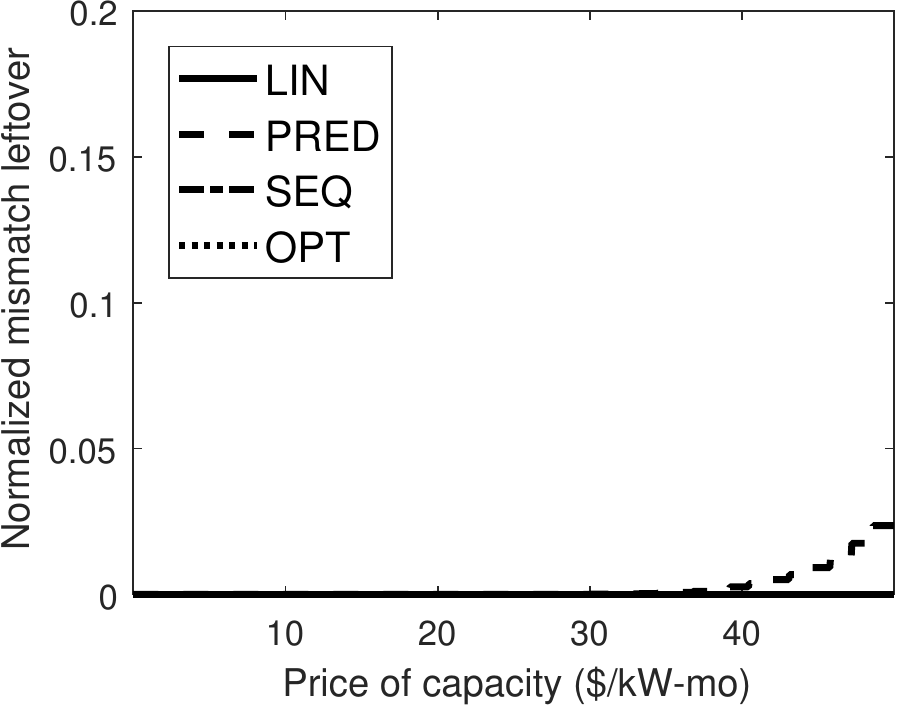}}
			\label{FIG:COMPSCH_VF}}
		\vspace{-0.15in}
		\caption{Comparing the PRED and LIN with OPT and SEQ for different capacity prices. The Demand Response (c) and Capacity Mismatch (d) are normalized by the average deviation of $D$.}
		\label{f.COMPSCH}
	\end{center}
	\vspace{-0.05in}
\end{figure*}

We compare our proposed algorithms PRED and LIN with two baselines: the offline optimal solution OPT as the lower bound for social cost and the sequential algorithm SEQ as the upper bound. 

\vspace{-0.05in}
\subsection{Experimental Setup}
\revised{We aim to use realistic parameters in the experimental setup to evaluate the performance of our proposed algorithms, and to understand the solution properties of each algorithm.
We model an LSE supplying power to 300 customers with a demand response timeslot that is five minutes long.}
\removed{We aim to use realistic parameters in the experimental setup to evaluate the performance of our proposed algorithms, and to understand the properties of solutions to each algorithm.
We model an LSE supplying power to 300 customers and a demand response timeslot that is five minutes long.}


\revised{The LSE first purchases capacity for which we model the cost as a linear function $c\kappa$ with the parameter $c\in\$[0.01,50]$/kW-mo to cover a broad range of grid capacity type costs and amortized storage.
The generation cost function for the LSE is modeled as a quadratic function $A\left(D-\sum_ix_i\right)^2$ with the parameter $A=\$0.1/12^2/\text{kW}^2$.}
\removed{The LSE must first purchase capacity for which we model the cost as a linear function $c\kappa$ with a cost parameter $c\in\$[0.01,50]$/kW-mo to observe over a broad range of grid capacity type costs and amortized grid storage.
The generation cost function for the LSE is modeled as a quadratic function $A\left(D-\sum_ix_i\right)^2$ with a cost parameter of $A=\$0.1/12^2/\text{kW}^2$.}
For this cost function setting, a deviation of 60kW for five minutes is equivalent to an energy cost of \$0.50/kWh and matches the intuition that larger mismatches are increasingly more expensive to manage.

\revised{To model each customer's particular load demand, we utilize the traces obtained from the UMass Trace Repository which give very granular load measurements from three homes (cf. Figure \ref{FIG:HOMESABCEMPCDF} for cumulative distributions)~\cite{barker2012smart}.
To make more accurate general conclusions we need more customers' load data.
However, the focus of this evaluation is to compare the performance of the different algorithms presented and demonstrate how to employ customer data to them.
We model the cost incurred by each customer to change its consumption as a quadratic cost $a_ix_i^2$ with the parameter $a_i\in\$[1,10]/12^2\text{kW}^2$.
Under these settings, a consumption decrease of 0.3kW for five minutes would cost the customer an energy price equivalent to \$0.025-0.25/kWh.}
\removed{Each customer has a particular demand of load.
To model this we utilize the traces obtained from the UMass Trace Repository which give very granular load measurements from three homes (cf. Figure \ref{FIG:HOMESABCEMPCDF} for cumulative distributions)~\cite{barker2012smart}.
To make more accurate general conclusions we need more customers' load data.
However, the focus of this evaluation is to compare the performance of the different algorithms presented and demonstrate how to employ customer data to them.We model the cost incurred by each customer to change its consumption as a quadratic cost $a_ix_i^2$ with a cost parameter $a_i\in\$[1,10]/12^2\text{kW}^2$.
With these settings, a consumption decrease of 0.3kW for five minutes would cost the customer an energy price equivalent to \$0.025-0.25/kWh.}
To generate customer cost uncertainties we first randomly choose $\tilde{a}_i$ from a bounded normal distribution. Each $\tilde{a}_i$ is used as a mean to further generate the training and test sets separately from a bounded normal distribution.  The $\hat{a}_i$ approximated by the LSE is then formed by averaging the training set values for customer $i$.


Renewable generation is incorporated into our simulations by using the ISO-NE's data on hourly wind power production for the same dates as the UMass data~\cite{isonefuelmix}.  A week-long sample and cumulative distribution for all the data are shown in Figure \ref{f.Wind_trace}. The capacity is set to 100kW by default, while we vary it in the sensitivity analysis.
The training and test sets for which we have trace data (Homes A,B,C and ISO-NE wind production) were made by first randomly separate each trace into a raw training and raw test set.  For each set we generate customer traces by bootstrapping 100 customers from each of the UMass Homes A/B/C.  We also do this for the ISO-NE wind data which is first normalized by the maximum power output so that we can scale wind power.
The demand mismatch $D$ over two days and the corresponding CDF are depicted in Figure~\ref{f.D_trace}.

\begin{figure*}[!ht]
	\begin{center}
		\subfigure[Social Cost]{{\includegraphics[width=0.49\columnwidth]{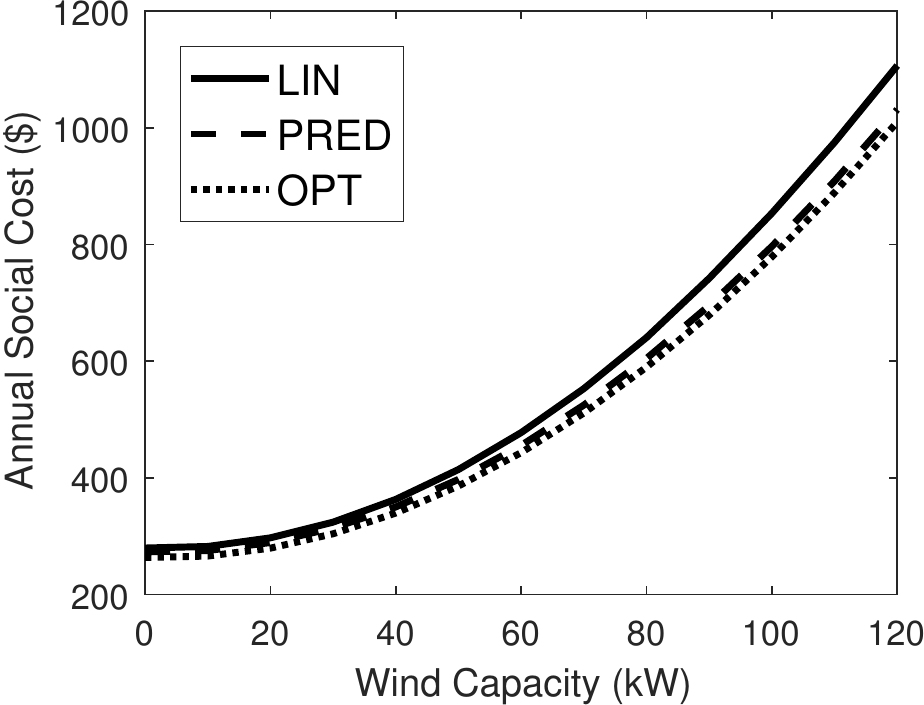}}
			\label{FIG:COMPSCH_SOC_W}}
		\subfigure[Optimal Capacity]{{\includegraphics[width=0.49\columnwidth]{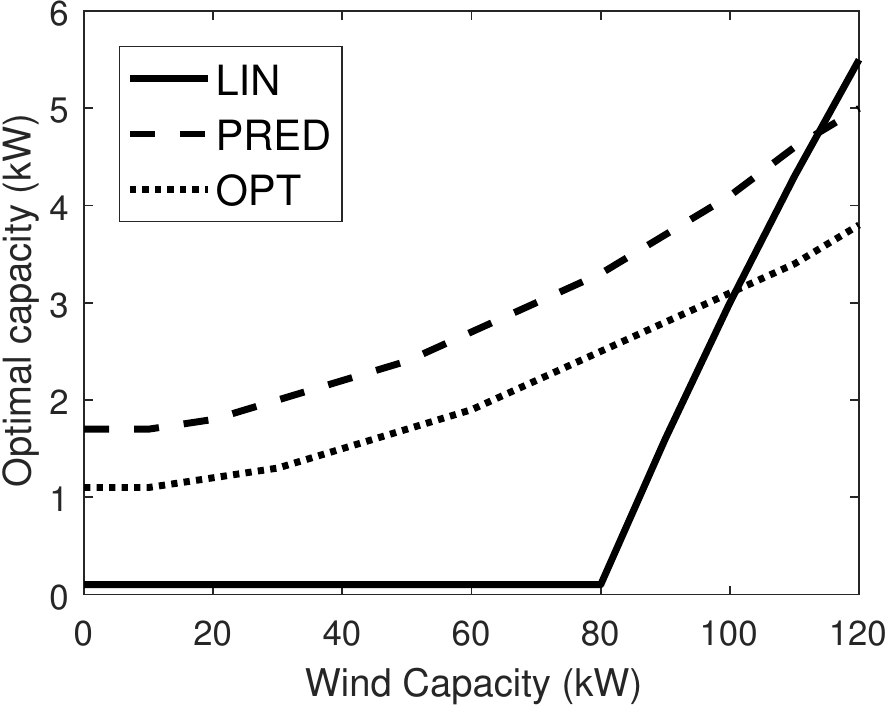}}
			\label{FIG:COMPSCH_OPTCAP_W}}
		\subfigure[Demand Response]{{\includegraphics[width=0.49\columnwidth]{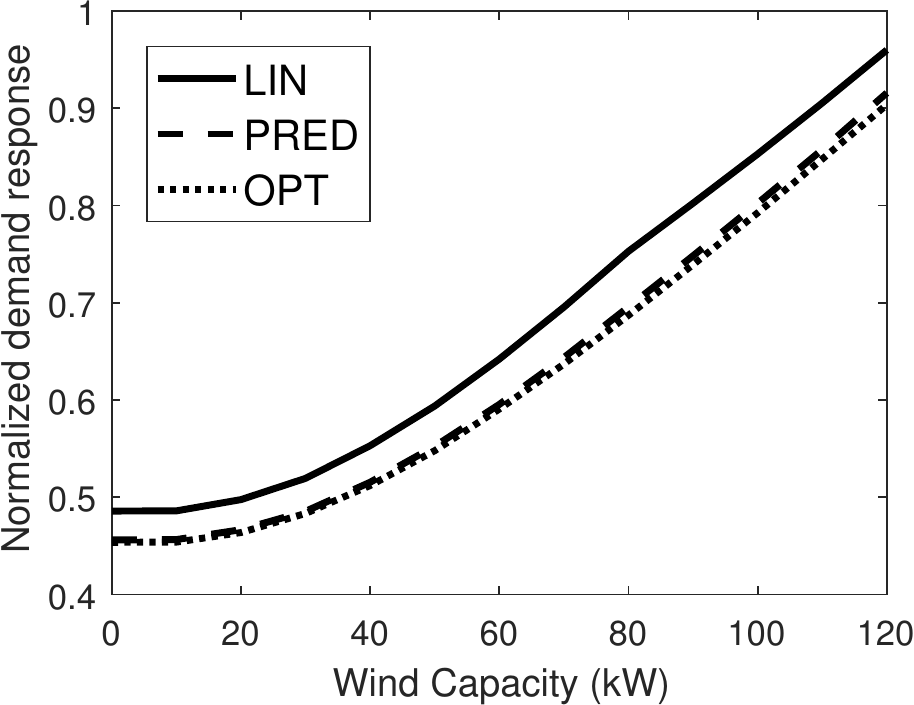}}
			\label{FIG:COMPSCH_DR_W}}
		\subfigure[Leftover Mismatch]{{\includegraphics[width=0.49\columnwidth]{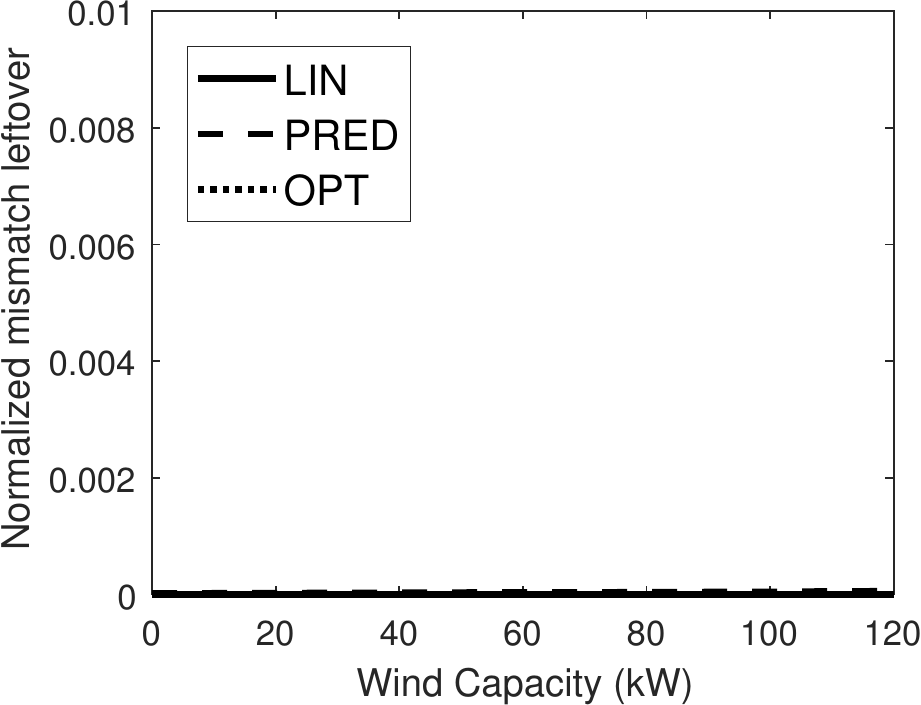}}
			\label{FIG:COMPSCH_VF_W}}
		\vspace{-0.15in}
		\caption{Comparing the Linear Policy and Price Prediction Scheme with the Offline optimal solution and Sequential policy for different amounts of wind power capacity.  The cost of capacity was set to \$1/kW-mo.  The Demand Response (b) and Capacity Mismatch (d) are normalized by the average deviation of $D$.}
		\label{f.COMPSCH_W}
	\end{center}
	\vspace{-0.05in}
\end{figure*}

\begin{figure}
	\begin{center}
		\includegraphics[width=0.6\columnwidth]{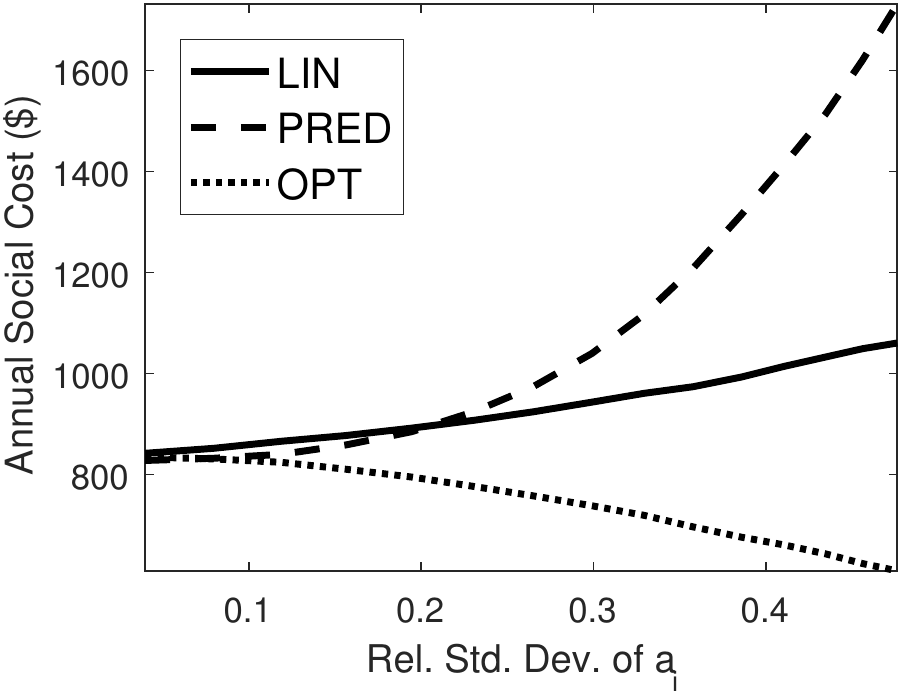}
		\vspace{-0.1in}
		\caption{Comparing the annual social cost with different Relative Standard Deviations (RSD) of the customer cost parameters $a_i$.  The cost of capacity was set to \$10/kW-mo.}
		\label{f.COMPSCH_AVAR}
	\end{center}
	\vspace{-0.05in}
\end{figure}

\vspace{-0.05in}
\subsection{Performance evaluation}
The evaluation of our algorithms and the cost savings will be organized around the following topics.

\vspace{-0.05in}
\subsubsection*{Convergence of the distributed algorithm}


We start by considering the convergence of our distributed algorithm for LIN.
\revised{Figure~\ref{FIG:DISTALG_SOCCOST} illustrates that the social cost of the distributed algorithm converges quickly to that of the centralized algorithm which validates the convergence analysis of Theorem \ref{th:convergence}.
For the parameters, even if we start with $\alpha_i=0$, it quickly converges to the optimal $\alpha_i$ which are different among customers as they have different parameters, i.e., $a_i$ in their cost functions.
For the other LIN parameters, $\gamma_i$ and $\beta_i$ stay at zero across all customers since the optimal $\gamma_i=0$ for zero-mean deviations and optimal $\beta_i=0$ for quadratic cost functions.
Thus, the remaining plots show only the centralized algorithm.}
\removed{Figure~\ref{FIG:DISTALG_SOCCOST} illustrates that the social cost of the distributed algorithm converges quickly to that of the centralized algorithm.
It validates the convergence analysis for the distributed algorithm.
For the parameters, even if we start with $\alpha_i=0$, it quickly converges to the optimal $\alpha_i$ and stays there.
In particular, the $\alpha_i$'s are different among customers as they have different parameters, i.e., $a_i$ in their cost functions.
On the other hand, $\gamma_i$ stays at zero across all customers, consistent with our analysis.
The situation is similar to $\beta$ as the optimal value is $0$ for quadratic cost functions.
Therefore, the remaining of the plots show only the centralized algorithm.}

\vspace{-0.05in}
\subsubsection*{How well do PRED and LIN perform?}


We move to this key question of the evaluation. 
To evaluate the benefits in terms of social cost savings for our algorithms, Figure~\ref{f.COMPSCH} compares the social cost of PRED and LIN to those under the offline optimal OPT and sequential algorithm SEQ.
\revised{Since OPT requires knowledge about each realization of the parameters, i.e., the exact renewable generation, customers' demand and cost functions at each timeslot, it is not practically feasible.}
\removed{Recall that OPT requires knowledge about each realization of the parameters, i.e., the exact renewable generation, customers' demand and cost functions at each timeslot, and therefore not practically feasible.}
We simply use this as a lower bound of the social cost.
However, SEQ makes conservative capacity planning decision about $\kappa$ first, and then performs similar to PRED.
By comparing to the cost of SEQ, the benefit of joint optimization of capacity planning on $\kappa$ and real-time demand response is highlighted.

The social costs of PRED and LIN, shown in Figure~\ref{FIG:COMPSCH_SOC}, are no more than $10\%$ higher compared to the fundamental limit OPT. 
\removed{This highlights that these two algorithms are highly effective.}
Importantly, when the capacity price decreases, the gap between PRED/LIN and OPT becomes smaller.
Since currently the technology improvements are reducing the capacity price, 
from Figure~\ref{FIG:COMPSCH_SOC}, this gap 
will decrease further.
The social cost of SEQ increases rapidly with increasing capacity prices because of the conservative 90kW capacity used to protect the system from \emph{any} leftover mismatch.

When capacity price increases, all algorithms have higher costs as expected.
However, PRED increases faster than LIN.
\revised{This is because LIN starts to rely more on the linear contracts than capacity, so it becomes more immune to the capacity price increases.
This provides the first guideline to choose between PRED and LIN in practice.}
\removed{This is because LIN starts to rely on the demand response as there is a linear contract, so it becomes less sensitive and more immune to the price increase.
This actually provides the first guideline for how to pick between PRED and LIN in practice.}
The second guideline depends on the level of customers' uncertainties, and will be discussed soon.

Interestingly, for PRED, LIN, and OPT, the social cost flatten when the price exceeds some threshold.
This is because when capacity price is too high, the LSE would purchase little capacity, as shown in Figure~\ref{FIG:COMPSCH_OPTCAP} and pay more to get demand response from customers, illustrated by Figure~\ref{FIG:COMPSCH_DR}.
\revised{The downside in the case of PRED, however, is the increased remaining mismatch that is not balanced by demand response, as shown in Figure~\ref{FIG:COMPSCH_VF}.}
\removed{The downside, however, is the increased remaining mismatch that is not balanced by demand response, as shown in Figure~\ref{FIG:COMPSCH_VF}.}

To understand more behind the social costs, Figures~\ref{FIG:COMPSCH_OPTCAP} and \ref{FIG:COMPSCH_DR} illustrate the optimal capacities and the amount of demand response extracted for all algorithms.
\revised{As the price increases, the capacities decrease, and the amounts of extracted DR increase to compensate for the reduced reserve capabilities.}
\removed{As the price increases, the capacities decrease, and the amounts of extracted DR increase to compensate for the reduced reserve capabilities due to lower capacities.}
Note in Figure~\ref{FIG:COMPSCH_OPTCAP}, PRED always has a higher capacity than OPT because PRED needs some additional capability to handle the estimation errors.
\revised{Similarly, when setting the price, PRED picks a higher price to extract more demand response on average, as shown in Figure~\ref{FIG:COMPSCH_DR}, to act as a buffer in case some customers' responses are overestimated.}
\removed{Similarly, when setting the price, PRED picks a higher price to extract more demand response on average, as shown in Figure~\ref{FIG:COMPSCH_DR}, so there is some room if some customers' responses are overestimated.}
The increased cost due to uncertainties in PRED compared to OPT can be considered as a ``risk premium''.
We will vary the customers' uncertainties to provide more information in Figure~\ref{f.COMPSCH_AVAR}.


\begin{figure}[!ht]
	\begin{center}
		\subfigure[Social Cost]{{\includegraphics[width=0.4\columnwidth]{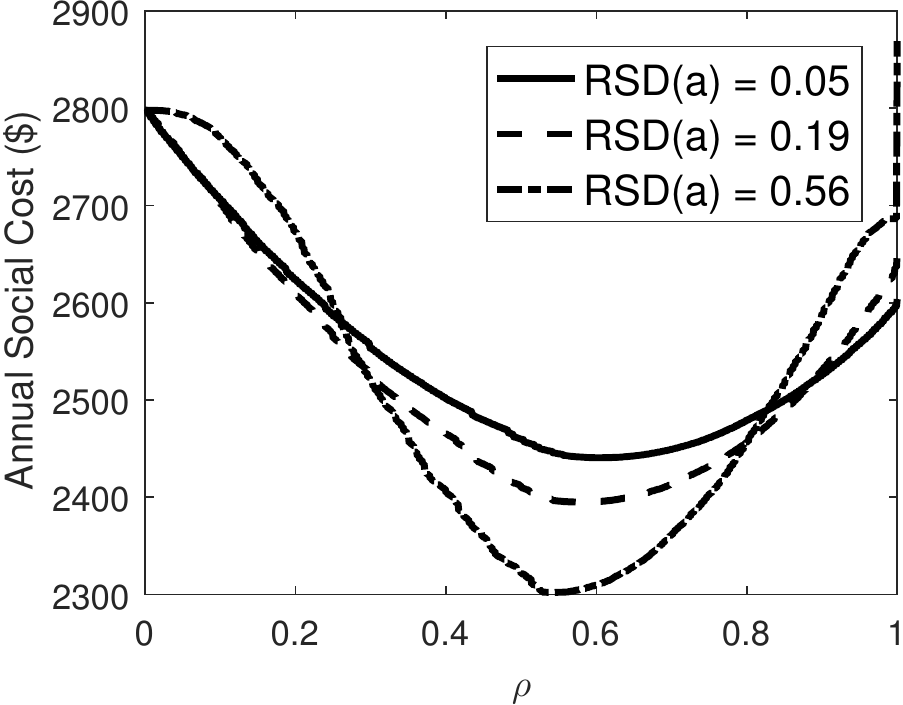}}
			\label{FIG:FDR_SOC}}
		\subfigure[Leftover Mismatch]{{\includegraphics[width=0.4\columnwidth]{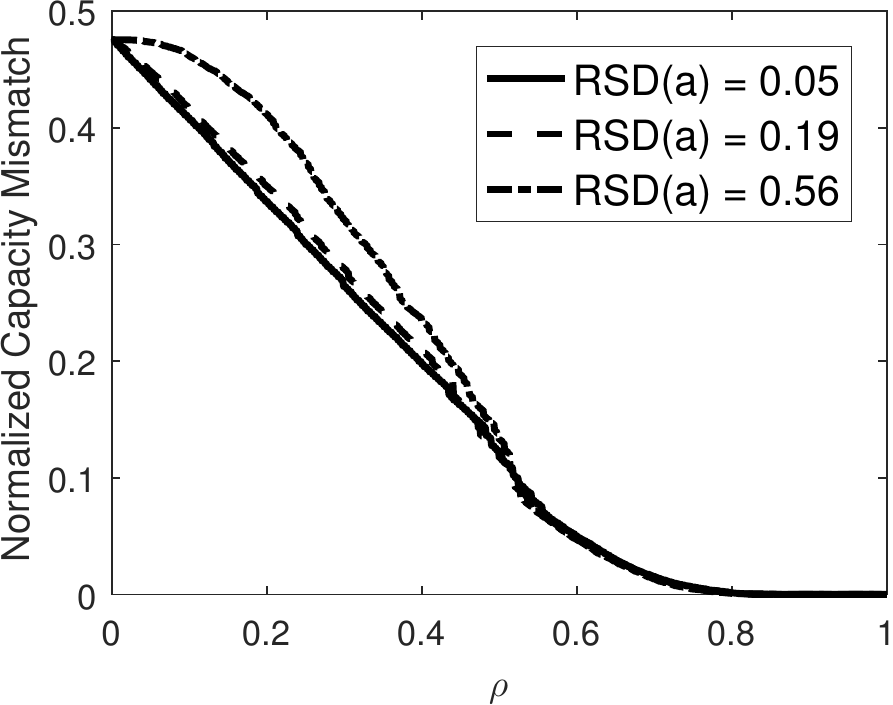}}
			\label{FIG:FDR_VF}}
		\vspace{-0.2in}
		\caption{The effect of the level of commitment on (a) the annual social cost and (b) the normalized capacity mismatch for different amounts of Relative Standard Deviations (RSD) on the customer cost parameter $a$.}
		\label{f.FDR}
	\end{center}
	\vspace{-0.15in}
\end{figure}

\vspace{-0.1in}
\subsubsection*{Impacts of the renewable penetration level.}
\revised{As renewable energy installation is rapidly growing, we evaluate the impacts of installed renewable energy capacities.}
\removed{As renewable energy installations are rapidly increasing, we evaluate the impacts of installed renewable energy capacities.}
Illustrated by Figure~\ref{f.COMPSCH_W}, social cost increases with more renewable energy installed as there is larger mismatch to be balanced.
When the wind capacity is too large, LIN can no longer just rely on the demand response, but needs a higher capacity, as shown in Figure~\ref{FIG:COMPSCH_OPTCAP_W}.
Throughout the figures, our proposed algorithms PRED and LIN perform very closely to the offline OPT.

\vspace{-0.05in}
\subsubsection*{Impacts of customer uncertainties.}
As another key message of the paper, we investigate the impact of the customer cost function uncertainties through the parameter $a_i$.
Here we fix the average cost parameter $\mathbb{E}[a_i]$ and vary the variance of $a_i$.

Figure~\ref{f.COMPSCH_AVAR} illustrates the social cost of the algorithms under different relative standard deviations.
Interestingly, as the variance increases, OPT actually has a lower cost.
The reason is that since OPT is the offline solution, it can utilize the realization of each $a_i(t)$ to target the customers with small cost parameters for each timeslot $t$.
As the variance increases, this targeted set of customers will have lower parameters, which leads to lower social cost in OPT.

PRED has the same cost as OPT when the variance is 0, but the gap becomes larger as the variance increases.
Unlike the case of increasing uncertainty from larger renewable capacities, the increased uncertainties due to a larger variance in $a_i$ cannot be well handled by PRED.
However, the cost of LIN does not change much. 
This suggests the second rule to pick between PRED and LIN: with larger uncertainties, LIN performs relatively better.


\vspace{-0.05in}
\subsection{Flexible Commitment Demand Response}
\label{sec:fcdr}
\vspace{-0.05in}
\subsubsection{LIN with flexibility: LIN$^+(\rho)$}

\revised{One potential drawback of LIN is that customers are forced to follow the specified linear policy and may from time-to-time face a very high cost to follow the policy.
From the social perspective, it is not beneficial to force a customer under much higher than average costs to provide demand response as compared to the LSE cost for tolerating more mismatch.}
\removed{One potential drawback of LIN is that customers are forced to follow the linear policy specified.
In some cases, customers may face a very high cost to follow the policy, e.g., when there are some critical jobs to be finished, represented by a larger $a_i(t)$.
From the social perspective, it is not beneficial to force a customer under much higher than average costs to provide demand response, as this would force them to incur higher costs compared to the LSE tolerating more mismatch.}

Motivated by this observation and the regulation service program described in Section~\ref{sec:bg}, we modify the LIN policy to add some flexibility limited by a single parameter $\rho$.
We call the new algorithm LIN$^+(\rho)$. 
Under LIN$^+(\rho)$, each customer has up to $1-\rho$ (in percentage) of the timeslots in which they do not need to follow the policy according to her realized $\alpha_i(t)$.
In other words, she may pick her original $d_i(t)$ with $x_i=0(t)$ for such timeslots.


\revised{A caveat is that we add is to only allow the customers to pick such timeslots according to their cost functions, e.g., through the parameters $\alpha_i(t)$ regardless of $D(t)$.
From the social perspective, allowing customers with higher $\alpha_i(t)$ to violate results in lower social costs.
However if the customers had full freedom, most customers would choose to violate when $D(t)$ is high which is when the LSE needs DR the most.
In reality, this condition can be enforced by calculating $\mathbb{E}_i[D|i\;\text{violates}]$ for the timeslots that a particular customer chooses to violate.}
\removed{Another caveat is that we only allow the customers to pick such timeslots according to their cost functions, e.g., through the parameters $\alpha_i(t)$, instead of the real-time $D(t)$.
The reason is that from the social perspective, allowing the customers with higher $\alpha_i(t)$ to violate results in lower social costs, but a higher $D(t)$ is seen by all customers and this is when the LSE needs DR the most.
Allowing some to violate means others need to take more, so there is a negative externality here.
From classical economic theory, this may lead to significant efficiency loss, or formally, the price of anarchy.
In reality, this can be enforced by calculating $\mathbb{E}_i[D|i\;\text{violates}]$ for the timeslots that a particular customer chooses to violate.}
If this value is significantly higher than $\mathbb{E}[D]$, then the customer faces some penalty.

Note that although we add the flexibility to LIN in this paper, the approach is in fact general and can be applied to a wide range of fully committed programs as follows.
Each customer is allowed to violate her commitment by up to $1-\rho$ (in percentage) of the timeslots.
A customer can run a local optimization to decide which timeslots to violate, while the LSE can add some constraints to align the local optimization with the social optimization, such as those described above.


\vspace{-0.1in}
\subsubsection{Cost versus mismatch leftover tradeoff}
\label{sec:fcdr-evaluation}

Now we evaluate the additional cost savings brought by LIN$^+(\rho)$.
Figure~\ref{f.FDR} highlights the tradeoff between cost and mismatch leftover that cannot be balanced due to capacity constraints.
Recall that there is little mismatch leftover in LIN, or equivalently, LIN$^+$(1).

Depicted in Figure~\ref{FIG:FDR_SOC}, as $\rho$ decreases from 1, the social cost first decreases due to the fact that some customers with very high $a_i(t)$ are allowed to not provide demand response.
As $\rho$ continues to decrease, we have more customers not providing demand response and the cost actually goes up again.
This is because the penalty for the mismatch becomes larger than the costs of customers to provide demand response.

On the other hand, Figure~\ref{FIG:FDR_VF} highlights that the mismatch leftover increases with lower $\rho$ as more customers are allowed to not provide demand response.
\revised{Importantly, the increase is very slow at the beginning when we decrease $\rho$ from $1$ because there is enough capacity to handle the small deficit in demand response.}
\removed{Importantly, the increase is very slow at the beginning when we decrease $\rho$ from $1$ because there is enough capacity to handle the small deficit in demand response when $1-\rho$ is small.}
However, the social cost decreases a lot.
This actually highlights the great potential to decrease $\rho$ appropriately to achieve a lower social cost but still have little mismatch leftover.
For instance, in our case study shown in Figure~\ref{f.FDR}, $\rho=0.8$ achieves cost savings 7-8\% with less than 1\% of mismatch leftover.
Recall that the gap between LIN and the offline optimal OPT is about 10\%. 
This means LIN$^+(\rho^*)$ achieves near optimal cost.



\vspace{-0.1in}
\section{Conclusion}
\vspace{-0.05in}
Extracting reliable demand response from customers is crucial to reduce/defer new energy storage/reserve installation and has significant environmental benefits.
However, existing demand response programs suffer from either low participation due to strict commitment, or not being reliable in voluntary programs.
Moreover, the capacity planning for energy storage/reserve is traditionally done without considering the demand response capabilities, which incurs inefficiencies.
Additionally, the uncertainties on the customers costs in providing demand response receive little attention in literature.
In this paper, we first model the problem as a stochastic optimization problem, and then design two online algorithms, PRED and LIN, to jointly optimize the capacity and demand response program design. 
PRED utilizes historical data to estimate the cost functions, while LIN is a linear contract between the LSE and customers.
We further design a distributed algorithm with guaranteed convergence for obtaining the optimal parameters for LIN. 
In addition, we design a flexible commitment demand response program LIN$^+(\rho)$ by adding flexibility to LIN.
Numerical simulations highlight the near optimal performance of the proposed algorithms, and large cost savings compared to a widely used baseline SEQ especially when capacity price is not negligible. 

Regarding further directions, the most exciting one is to incorporate power network constraints into our optimization framework. 
This will bring some new challenges and insights. 
In particular, the parameter $\beta_i$ for the customer's own prediction error in LIN will no longer be zero due to the fact that errors can be handled easier locally because of the loss over power lines and the line capacity constraints.
\vspace{-0.1in}
\section{Acknowledgments}
\vspace{-0.05in}
We thank all of the reviewers and our shepherd Shaolei Ren for their valuable feedback. This research is supported by NSF grants CNS-1464388 and CNS-1617698. 
\vspace{-0.05in}

\bibliographystyle{ACM-Reference-Format}
\bibliography{reference}


\appendix
\section*{Appendix}
\section{Karush-Kuhn-Tucker conditions}
\label{sec:KKT}

Karush-Kuhn-Tucker (KKT) conditions for optimality of $\bx(t)$ at time $t$ in Problem \eqref{opt:OFFLINE1_REALTIME} are:
\vspace{-0.05in}
\begin{subequations}\label{eq:KKTgen2}
	\begin{equation}\label{eq:KKTgen2stat}
		C'_i(x_i^*(t))-C'_\text{g}\left(D(t)-\sum_{j\in\mathcal{V}}x_j^*(t)\right)+\underline{\theta}^*-\overline{\theta}^*=0,\quad\forall i\in\mathcal{V}
	\end{equation}
	\begin{equation}\label{eq:KKTgen2comp2}
		\underline{\theta}^*\left(D(t)-\sum_{j\in\mathcal{V}}x_j^*(t)+\kappa\right)=0
	\end{equation}
	\begin{equation}\label{eq:KKTgen2comp}
		\overline{\theta}^*\left(D(t)-\sum_{j\in\mathcal{V}}x_j^*(t)-\kappa\right)=0
	\end{equation}
	\begin{equation}\label{eq:KKTgen2dual}
		\underline{\theta}^* \geq 0,\quad \overline{\theta}^* \geq 0
	\end{equation}
	\begin{equation}\label{eq:KKTgen2prim}
		-\kappa\leq D(t)-\sum_{j\in\mathcal{V}}x_j^*(t)\leq\kappa
	\end{equation}
\end{subequations}
where $(\underline{\theta},\overline{\theta})$ are the dual variables for the constraint \eqref{const:Capacity_RT}, \eqref{eq:KKTgen2stat} are the first-order stationary conditions, \eqref{eq:KKTgen2comp} and \eqref{eq:KKTgen2comp2} are the complementary slackness conditions, \eqref{eq:KKTgen2dual} are the dual feasibility conditions, and \eqref{eq:KKTgen2prim} are the primal feasibility conditions.

To better understand the physical implications of these conditions, we start from \eqref{eq:KKTgen2comp2} and \eqref{eq:KKTgen2comp}. 
When the constraint \eqref{eq:KKTgen2prim} on $\kappa$ is non-binding, i.e., $-\kappa<D(t)-\sum_{j\in\mathcal{V}}x_j^*(t)<\kappa$, we have $\underline{\theta}^*=\overline{\theta}^*=0$.
Then \eqref{eq:KKTgen2stat} implies $C'_i(x_i^*(t))=C'_\text{g}\left(D(t)-\sum_{j\in\mathcal{V}}x_j^*(t)\right)$, meaning that the marginal cost for each customer to provide demand response is the same, all of which is equal to the LSE's marginal cost to tolerate the mismatch.

The KKT condition of optimality for capacity $\kappa$ in Problem \eqref{opt:OFFLINE1} is:
\begin{equation}
	C'_\text{cap}(\kappa^*)+\frac{\partial}{\partial\kappa}\mathbb{E}_{\delta_i,\delta_r}\left[R(\kappa^*;t)\right]=0
\end{equation}
where $R(\kappa;t)$ is the optimal value of Problem \eqref{opt:OFFLINE1_REALTIME}.
We can interchange the derivative operator and expectation operator if either of the following conditions are true:
(i) $(\delta_i,\delta_r)$ are from continuous probability distributions and $R(\kappa;t)$ is continuously differentiable for all $(\delta_i,\delta_r)$ (from the Leibniz Integral Rule); (ii) $(\delta_i,\delta_r)$ are from discrete probability distributions (differentiation is a linear operator).  This results in the following:
\begin{equation}
	C'_\text{cap}(\kappa^*)+\mathbb{E}_{\delta_i,\delta_r}\left[\frac{\partial}{\partial\kappa}R(\kappa^*;t)\right]=0
\end{equation}

From Lemma \ref{lm:convexKAPPA}, the negative of the sum of the dual variables $\underline{\theta}+\overline{\theta}$ for constraint \eqref{const:Capacity_RT} is the gradient of $R(\kappa;t)$ w.r.t. $\kappa$. This allows us to substitute $\underline{\theta}^*+\overline{\theta}^*$ for the partial derivative:
\begin{equation}\label{eq:POSToptKappa}
	C'_\text{cap}(\kappa^*) = \mathbb{E}_{\delta_i,\delta_r}\left[\theta(\kappa^*;t)\right]
\end{equation}
where we use the notation of $\theta(\kappa;t)$ as a function to represent the sum of the optimal dual variables $\underline{\theta}^*+\overline{\theta}^*$ from the KKT conditions \eqref{eq:KKTgen2} for given $(\kappa;t)$.

\section{Proofs}
\label{sec:app-1}

\begin{proof}[Proof of Theorem 5]
	(Based on \cite{boyd2014subgradient} and \cite{bertsekas1999nonlinear} Chapter 6.)  At iteration $k$, let $\Lambda^{(k)}:=(\boldsymbol{\pi}^{(k)},\boldsymbol{\lambda}^{(k)},\boldsymbol{\mu}^{(k)})$ be the current dual prices, $Q(\Lambda^{(k)})$ be the partial dual function value of \eqref{opt:GENall2D}, $G^{(k)}$ be the subgradient of the dual function, and $\eta^{(k)}$ be the step size.  Also note that the subgradient at iteration k for constraint \eqref{const:SEPARATE} is $G^{(k)}:=(\boldsymbol{\alpha}^{(k)},\boldsymbol{\beta}^{(k)},\boldsymbol{\gamma}^{(k)})-(\mathbf{u}^{(k)},\mathbf{v}^{(k)},\mathbf{w}^{(k)})$.
	
	Start with the squared Euclidean distance between the iteration $(k+1)$ and optimal dual price vectors and substitute $\Lambda^{(k+1)}$ with the dual price update in Equation \eqref{step:DUALupdate}:
	\begin{equation}
	||\Lambda^{(k+1)}-\Lambda^*||_2^2=||\Lambda^{(k)}+\eta^{(k)}G^{(k)}-\Lambda^*||_2^2.
	\end{equation}
	Expanding out the terms on the RHS becomes:
	\begin{equation}
		||\Lambda^{(k+1)}-\Lambda^*||_2^2\nonumber
	\end{equation}
	\begin{equation}
		=||\Lambda^{(k)}-\Lambda^*||_2^2-2\eta^{(k)}G^{(k)\text{T}}(\Lambda^*-\Lambda^{(k)})+(\eta^{(k)})^2||G^{(k)}||_2^2.
	\end{equation}
	The first-order condition of the subgradient is $G^{(k)\text{T}}(\Lambda^*-\Lambda^{(k)})\geq Q(\Lambda^*)-Q(\Lambda^{(k)})$ which when applied to the equation above gives us:
	\begin{equation}
	||\Lambda^{(k+1)}-\Lambda^*||_2^2\nonumber
	\end{equation}
	\begin{equation}
	\leq||\Lambda^{(k)}-\Lambda^*||_2^2-2\eta^{(k)}(Q(\Lambda^*)-Q(\Lambda^{(k)}))+(\eta^{(k)})^2||G^{(k)}||_2^2.
	\end{equation}
	Apply the above inequality recursively from iteration $k$ to $1$ which results in:
	\begin{equation}
	||\Lambda^{(k+1)}-\Lambda^*||_2^2\nonumber
	\end{equation}
	\begin{equation}
	\leq||\Lambda^{(1)}-\Lambda^*||_2^2-2\sum_{i=1}^{k}\eta^{(i)}(Q(\Lambda^*)-Q(\Lambda^{(i)}))+\sum_{i=1}^{k}(\eta^{(i)})^2||G^{(i)}||_2^2.
	\end{equation}
	Since $||\Lambda^{(k+1)}-\Lambda^*||_2^2\geq 0$, then the LHS of the above inequality is lower bounded by zero and the middle term of the RHS can be brought to the other side:
	\begin{equation}
	2\sum_{i=1}^{k}\eta^{(i)}(Q(\Lambda^*)-Q(\Lambda^{(i)}))\leq||\Lambda^{(1)}-\Lambda^*||_2^2+\sum_{i=1}^{k}(\eta^{(i)})^2||G^{(i)}||_2^2.
	\end{equation}
	Let us define $Q(\Lambda_\text{best}^{(k)}):=\max_{i=1,...,k}Q(\Lambda^{(i)})$ as the best dual value so far in $k$ iterations.  This gives a new lower bound to the LHS of the above inequality and allows the summation to be factored out:
	\begin{equation}
	2(Q(\Lambda^*)-Q(\Lambda_\text{best}^{(k)}))\sum_{i=1}^{k}\eta^{(i)}\leq||\Lambda^{(1)}-\Lambda^*||_2^2+\sum_{i=1}^{k}(\eta^{(i)})^2||G^{(i)}||_2^2.
	\end{equation}
	and can be brought to the RHS:
	\begin{equation}
	Q(\Lambda^*)-Q(\Lambda_\text{best}^{(k)})\leq\dfrac{||\Lambda^{(1)}-\Lambda^*||_2^2+\sum_{i=1}^{k}(\eta^{(i)})^2||G^{(i)}||_2^2}{2\sum_{i=1}^{k}\eta^{(i)}}.
	\end{equation}
	Apply the definition of the step size $\eta^{(i)}:=\frac{\zeta/i}{||G^{(i)}||_2}$ from \eqref{step:STEPupdate} which turns the above inequality into:
	\begin{equation}
	Q(\Lambda^*)-Q(\Lambda_\text{best}^{(k)})\leq\dfrac{||\Lambda^{(1)}-\Lambda^*||_2^2+\zeta^2\sum_{i=1}^{k}1/i^2}{2\zeta\sum_{i=1}^{k}\frac{1/i}{||G^{(i)}||_2}}.
	\end{equation}
	The Assumptions \ref{as:DUALbound} and \ref{as:SUBGRADbound} give an upper bound on the RHS on the above inequality and allows $\overline{G}$ to be factored out:
	\begin{equation}
	Q(\Lambda^*)-Q(\Lambda_\text{best}^{(k)})\leq\overline{G}\dfrac{\overline{\Lambda}^2+\zeta^2\sum_{i=1}^{k}1/i^2}{2\zeta\sum_{i=1}^{k}1/i}.
	\end{equation}
	Taking the limit as number of iterations $k$ approaches infinity makes the RHS approach zero.  Also since the LHS is lower bounded by zero, this results in the best dual value converging to the optimal dual value.
	Since dual the function $Q(\Lambda)$ is continuous, then $\lim\limits_{k\rightarrow\infty}\Lambda_\text{best}^{(k)}=\Lambda^{*}$.
\end{proof}

In ensuring the existence of the optimal solution to \eqref{opt:GENall2D}, we state the following theorem. 

\begin{definition}\label{df:EQL}
	The prices $(\boldsymbol{\pi},\boldsymbol{\lambda},\boldsymbol{\mu})$, LSE DR settings\\ $(\boldsymbol{\alpha},\boldsymbol{\beta},\boldsymbol{\gamma})$, and customer DR settings $(\mathbf{u},\mathbf{v},\mathbf{w})$ are in equilibrium if: (i) $(\boldsymbol{\alpha},\boldsymbol{\beta},\boldsymbol{\gamma})$ are optimal to \eqref{opt:LSEalg2} with prices $(\boldsymbol{\pi},\boldsymbol{\lambda},\boldsymbol{\mu})$, (ii) For all $i\in\mathcal{V}$: $(u_i,v_i,w_i)$ are optimal to \eqref{opt:CUSTalg2} with prices $(\pi_i,\lambda_i,\mu_i)$, and (iii) $(\boldsymbol{\alpha},\boldsymbol{\beta},\boldsymbol{\gamma})=(\mathbf{u},\mathbf{v},\mathbf{w})$.
\end{definition}

\begin{theorem}
	There exists a set of equilibrium prices $(\boldsymbol{\pi},\boldsymbol{\lambda},\boldsymbol{\mu})$, LSE DR settings $(\boldsymbol{\alpha},\boldsymbol{\beta},\boldsymbol{\gamma})$, and customer DR settings $(\mathbf{u},\mathbf{v},\mathbf{w})$. Additionally, the set is optimal to \eqref{opt:GENall2D}.
\end{theorem}
\begin{proof}
	If one takes the Lagrangian of \eqref{opt:GENall2D}, it can be noticed that it equals the sum of LSE's Lagrangian of \eqref{opt:LSEalg2} and all the customers' Lagrangians of \eqref{opt:CUSTalg2}.
	Also, the set of primal feasible solutions of \eqref{opt:GENall2D} is a subset of primal feasible solutions of \eqref{opt:LSEalg2} and \eqref{opt:CUSTalg2} for each customer $i$.
	Additionally, the dual feasible set and complementary slackness constraints coming from \eqref{const:CAPpolicy1}\eqref{const:CAPpolicy2} applies to both \eqref{opt:GENall2D} and \eqref{opt:LSEalg2}.
	Therefore a solution that satisfies the Karush-Kuhn-Tucker (KKT) conditions of \eqref{opt:GENall2D} also simultaneously satisfies the KKT conditions of \eqref{opt:LSEalg2} and \eqref{opt:CUSTalg2}, with the condition that $(\boldsymbol{\alpha},\boldsymbol{\beta},\boldsymbol{\gamma})=(\mathbf{u},\mathbf{v},\mathbf{w})$.
	An equilibrium point from Definition \ref{df:EQL} must also satisfy KKT conditions of \eqref{opt:LSEalg2} and \eqref{opt:CUSTalg2}, with the condition that $(\boldsymbol{\alpha},\boldsymbol{\beta},\boldsymbol{\gamma})=(\mathbf{u},\mathbf{v},\mathbf{w})$.   Since \eqref{opt:GENall2D},\eqref{opt:LSEalg2}, and \eqref{opt:CUSTalg2} are convex optimization problems, the KKT conditions are both necessary and sufficient for optimality.  Since there clearly exists an optimal point of \eqref{opt:GENall2D}, then there must exist an equilibrium point.  It can be checked that the KKT conditions of \eqref{opt:LSEalg2} and \eqref{opt:CUSTalg2}, along with $(\boldsymbol{\alpha},\boldsymbol{\beta},\boldsymbol{\gamma})=(\mathbf{u},\mathbf{v},\mathbf{w})$ are equivalent to the KKT conditions of \eqref{opt:GENall2D}.  Therefore an equilibrium point is optimal to \eqref{opt:GENall2D}.
\end{proof}

\end{document}